\documentclass[a4paper,twoside,12pt,pdftex,dvipsnames]{article}

\usepackage[T1]{fontenc}
\usepackage[utf8]{inputenc}

\usepackage[british]{babel}
\usepackage{csquotes}
\usepackage[framemethod=TikZ,xcolor=dvipsnames]{mdframed}
\usepackage{amssymb,amsmath,amsthm,thmtools}
	\numberwithin{equation}{section}
\usepackage{mathtools}	
\usepackage[only,Ydown,mapsfrom]{stmaryrd}
\usepackage[backend=biber,style=alphabetic,sorting=nyt,backref=true,useprefix=true,maxalphanames=4,minalphanames=3,maxnames=4,minnames=3]{biblatex}
\usepackage[margin=1in,headheight=15pt]{geometry}
\usepackage{fancyhdr}
\usepackage{tikz-cd}
	\usetikzlibrary{decorations.pathmorphing}
\usepackage[vlined,linesnumbered]{algorithm2e}

	\SetCommentSty{mycommfont}

\usepackage{todonotes}
\usepackage{calrsfs} 
	
	\DeclareMathAlphabet{\mathcal}{OMS}{zplm}{m}{n} 
\usepackage{comment}
\usepackage{array}
\usepackage{xparse}
\usepackage{pbox}
\usepackage[colorlinks,linktoc=all]{hyperref}
\addto\extrasbritish{%
	\def\Snospace~{\S{}}
}

\title{Finite Class $2$ Nilpotent and Heisenberg Groups}
\author{Dávid R. Szabó\footnote{The project
		leading to this application has received funding from the European
		Research Council (ERC) under the European Union’s Horizon 2020
		research and innovation programme (grant agreement No 741420). 
		The author was supported by the
		National Research, Development and Innovation Office (NKFIH) Grant
		K138596.}}
\date{\today}

\allowdisplaybreaks

\makeatletter
\let\orgdescriptionlabel\descriptionlabel
\renewcommand*{\descriptionlabel}[1]{%
	\let\orglabel\label
	\let\label\@gobble
	\phantomsection
	\edef\@currentlabel{#1}%
	\let\label\orglabel
	\orgdescriptionlabel{#1}%
}
\makeatother

\definecolor{Cambridgeblue}{RGB}{133, 176, 154}
\definecolor{CambridgeblueDark}{RGB}{72, 112, 92}

\definecolor{Oxfordblue}{RGB}{0, 33, 71}

\definecolor{CEUblue}{RGB}{0,152,195}
\definecolor{CEUblueDark}{RGB}{0,152,195}

\definecolor{ELTEblue}{RGB}{1, 40, 81}
\definecolor{ELTEblueLight}{RGB}{1, 60, 123}
	\hypersetup{
	allcolors=ELTEblue
}

\declaretheoremstyle[
	spaceabove=0, 
	spacebelow=0,
	headfont=\normalfont\bfseries,
	notefont=\mdseries, 
	notebraces={(}{)},
	bodyfont=\itshape,
	postheadspace=5pt plus 1pt minus 1pt,
	mdframed={
		backgroundcolor=black!10, 
		linecolor=black!20, 
		linewidth=1pt,
		innertopmargin=4pt,
		roundcorner=5pt, 
		innerbottommargin=4pt, 
	} 
]{plainFilled}

\declaretheoremstyle[
spaceabove=0, 
spacebelow=0,
headfont=\normalfont\bfseries,
notefont=\mdseries, 
notebraces={(}{)},
bodyfont=\itshape,
postheadspace=5pt plus 1pt minus 1pt,
mdframed={
	linecolor=black!20, 
	linewidth=2pt,
	innertopmargin=6pt,
	roundcorner=5pt, 
	innerbottommargin=6pt, 
} 
]{plainBordered}

\tikzset{
	symbol/.style={
		draw=none,
		every to/.append style={
			edge node={node [sloped, allow upside down, auto=false]{$#1$}}}
	}
}
\tikzset{
	commutative diagrams/.cd,
	mono/.style={hook}, 
	epi/.style={two heads}, 
	iso/.style={mono, epi, "\sim"{sloped}}, 
	iso'/.style={mono, epi, "\sim"{',sloped}},
	mapsto/.style={|->},
	identity/.style={equal},
	inclusion/.style={mono, "\subseteq"{sloped}},
	linclusion/.style={mono, "\supseteq"'{sloped}},
	central/.style={"\scriptscriptstyle \bullet"{description, inner sep=-1pt, pos=0.5, sloped}},
	double central/.style={mono,"\diamond"{description,inner sep=0pt, pos=0.7, sloped}},
	cprod/.style={mono,"\cprod"{description,inner sep=-3pt, pos=0.5, sloped}},
	dashed/.style={densely dashed},
	dotted/.style={dash dot},
	inline/.style={cramped, sep=small},
	cube/.style={sep={3.5em,between origins}},
	row sep/none/.initial=0em,
	align/.style={row sep=none},
	:/.style={r,draw=none, ":"{font=,anchor=center}},
	::/.style={rr,draw=none, ":"{font=,anchor=center}},
	dots/.style={r,draw=none, "\dots"{font=,anchor=center}},
	oplus/.style={draw=none, "\oplus"{font=,anchor=center}},
	times/.style={draw=none, "\times"{font=,anchor=center}},
	label/.style={/tikz/column 1/.append style={
			column sep= 0.75em}},
	smallColSep/.style={/tikz/column #1/.append style={
				column sep = 0.25em}},
	left label/.style={label,/tikz/column 1/.append style={
			anchor=base east,
			"xy"}}
}

\newcommand{\email}[1]{\href{mailto:#1}{\nolinkurl{#1}}}

\declaretheorem[style=plain,name=Theorem,numberwithin=section]{thm}
\declaretheorem[style=plainFilled,name=Theorem,sibling=thm]{thm!}

\declaretheorem[style=plainFilled,name=Theorem]{thmMain}

\declaretheorem[style=plain,name=Proposition,sibling=thm]{prop}
\declaretheorem[style=plainBordered,name=Proposition,sibling=thm]{prop!}

\declaretheorem[style=plain,name=Lemma,sibling=thm]{lem}
\declaretheorem[style=plainBordered,name=Lemma,sibling=thm]{lem!}

\declaretheorem[style=plain,name=Corollary,sibling=thm]{cor}
\declaretheorem[style=plainBordered,name=Corollary,sibling=thm]{cor!}

\declaretheorem[style=plainBordered,name=Problem,sibling=thm]{prob!}

\declaretheorem[style=definition,name=Definition,sibling=thm]{defn}
	
\declaretheorem[style=definition,name=Example,sibling=thm]{exmp}

\declaretheorem[style=remark,name=Remark,sibling=thm]{rem}
\declaretheorem[style=remark,name=Note,sibling=thm]{note}

\newcommand{\Z}{\mathbb{Z}}

\newcommand{\F}{\mathbb{F}}

\renewcommand{\phi}{\varphi}
\renewcommand{\rho}{\varrho}

\newcommand{\CC}{{\mathcal{C}}}

\DeclareMathOperator{\Syl}{Syl}

\newcommand{\cbox}[2]{\pbox{#1}{\small\relax\ifvmode\centering\fi #2}}

\DeclareMathOperator{\SL}{SL}

\DeclareMathOperator{\Aut}{Aut}

\DeclareMathOperator{\Hom}{Hom}

\DeclareMathOperator{\ann}{ann}

\DeclareMathOperator{\im}{Im}
\DeclareMathOperator{\rank}{rk}

\newcommand{\id}{\mathrm{id}}

\DeclareMathOperator{\diag}{diag}

\newcommand{\isom}{\cong}

\newcommand{\cprod}{\mathbin{\circ}} 

\newcommand{\fprod}{\mathbin{\star}}

\NewDocumentCommand{\centralProduct}{m o m}{{#1}\mathbin{\cprod\IfNoValueTF{#2}{}{\!_{#2}}} {#3}}
\NewDocumentCommand{\centralProj}{o}{\mathrm{p}\IfNoValueTF{#1}{}{_{#1}}}
\NewDocumentCommand{\centralUniv}{o}{\cprod\IfNoValueTF{#1}{}{_{#1}}}

\NewDocumentCommand{\centralProductMax}{m o m}{{#1}\mathbin{\hat\cprod\IfNoValueTF{#2}{}{\!_{#2}}} {#3}}

\NewDocumentCommand{\HeisenbergMono}{o}{\iota\IfNoValueTF{#1}{}{_{#1}}}
\NewDocumentCommand{\HeisenbergEpi}{o}{\pi\IfNoValueTF{#1}{}{_{#1}}}

\NewDocumentCommand{\cbaExtInline}{o m o m o m}{
	\begin{tikzcd}[inline,ampersand replacement=\&]
		\IfNoValueF{#1}{#1:}
		1\ar[r] \& 
		#2\ar[r,central,mono,"\IfNoValueF{#3}{#3}"] \& 
		#4\ar[r,epi,"\IfNoValueF{#5}{#5}"] \& 
		#6\ar[r] \& 
		1
	\end{tikzcd}
}

\DeclareMathOperator{\HH}{H}

\newcommand{\embeds}{\hookrightarrow}
\newcommand{\xembeds}[1]{\xhookrightarrow{#1}}
\newcommand{\surjects}{\twoheadrightarrow}
\newcommand{\xsurjects}[2][]{%
	\xrightarrow[#1]{#2}\mathrel{\mkern-14mu}\rightarrow
}

\newcommand{\normal}{\lhd}
\DeclareMathOperator{\Center}{Z}

\newcommand{\generate}[1]{{\langle #1\rangle}}

\DeclareMathOperator{\lcm}{lcm}

\newcommand{\centralByAbelian}{central-by-abelian}

\newcommand{\alternatingFunctor}{\mathop{\mathcal{A}}}

\newcommand{\HeisenbergFunctor}{\mathop{\mathcal{H}}}

\newcommand{\maximalCBAProj}{\pi_{\mathcal{Z}}}

\DeclareMathOperator{\maxCBAfunctor}{\mathcal{Z}}

\newenvironment{summary}{\begin{quote}\small}{\end{quote}}
\newenvironment{smallpmatrix}{\left(\begin{smallmatrix}}{\end{smallmatrix}\right)}
\newcommand{\matrixH}[3]{\begin{smallpmatrix}1&#1&#3\\0&1&#2\\0&0&1\end{smallpmatrix}}
\newcommand{\matrixGroupH}[4]{\matrixH{#1}{#2}{#3}_{\!\!#4}}

\newcommand{\divides}{\bigm|}

\newcommand{\from}{\colon}
\newcommand{\leteq }{\coloneqq}

\newcommand{\arrowInline}[1]{\begin{tikzcd}[inline,ampersand replacement = \&]%
	\!\ar[r,{#1}] \& \!%
\end{tikzcd}}

\newcommand{\extraSPecialD}[1]{{\mathfrak{D}_{#1}}}
\newcommand{\extraSPecialQ}[1]{{\mathfrak{Q}_{#1}}}

\newcommand{\coset}[1]{\underline{\smash{#1}}}
\newcommand{\cosetRepresentative}[1]{\hat{#1}}

\newcommand{\extended}[1]{{#1}^\star}

\begin{document}
\maketitle

\begin{abstract}
	We present a structural description of finite nilpotent groups of class at most~ $2$ using a specified number of subdirect and central products of such groups that are $2$-generated. 
	As a corollary, we show that any such group is isomorphic to a subgroup of  a Heisenberg group satisfying certain properties. 
	
	The motivation for these results is of topological nature as they can be used   
	to give lower bounds to the nilpotently Jordan property of the birational automorphism group of varieties and the homeomorphism group of compact manifolds. 
\end{abstract}

\section{Introduction}

A finite non-abelian $p$-group $G$ is \emph{special} if its Frattini subgroup $\Phi(G)$, derived subgroup $G'$ and centre $\Center(G)$ all coincide and are isomorphic to $(\Z/p\Z)^r$ for some $r$. 
A special group $G$ is \emph{extra-special} if $r=1$. 
The structure of these groups is described by the following classical result.
\begin{thm}[{\cite[(4.16)/(ii), Theorem~4.18]{Suzuki2}}]\label{thm:specialPgroups}
		Every special $p$-group is a subdirect product of groups of the form: the central product of an extra-special $p$-group and an abelian group.
		Every extra-special $p$-group of order $p^{2n+1}$ is the central product of $n$ extra-special $p$-subgroups of order $p^3$. 
		For every prime $p$, there are exactly two extra-special groups of order $p^3$ (up to isomorphism).
\end{thm}

We present a generalisation to finite nilpotent groups of class at most~$2$. 
\begin{thmMain}\label{thm:mainStructure}\
	\begin{enumerate}
		\item\label{part:mainStructureSubdirect}
		Every finite nilpotent group $G$ is a subdirect product of $d(Z(G))$ groups each with cyclic centre. 
		No such subdirect product may exists using fewer factors.
		
		\item\label{part:mainStructureCentral}
		Every finite nilpotent group $G$ of class at most~$2$ with cyclic commutator subgroup is the internal central product of $t$ nilpotent $2$-generated subgroups of class $2$ and an abelian subgroup $A$ satisfying $d(G)=2t+d(A)$ and some further properties discussed in \autoref{cor:decompositionCyclicDeriverGroup}. 
	\end{enumerate}
\end{thmMain}
The $p$-groups of class $2$ with $2$ generators are classified in \cite[Theorem~1.1]{2generatedClassification}. 
The central product decomposition of \autoref{thm:mainStructure} is a generalisation of \cite[Theorem~2.1]{brady_bryce_cossey_1969} where groups with cyclic centre were considered. 
The argument presented in the current paper is more structural and gives some invariants needed for a topological application of \autoref{thm:mainAction} below.

For every $\Z$-bilinear map $\mu\from A\times B\to C$ between $\Z$-modules, we associate a (\emph{Heisenberg}) group 
\begin{equation}\label{eq:matrixGroupH}
	\matrixGroupH{A}{B}{C}{\mu}\leteq \left\{\matrixH{a}{b}{c}:a\in A,b\in B,c\in C\right\}
\end{equation}
where the group operation is formal matrix multiplication via $\mu$, see \autoref{def:H} and \autoref{rem:Hdef}.
As an application of (the proof of) \autoref{thm:mainStructure}, we obtain the other main statement of the paper.
\begin{thmMain}\label{thm:mainHeisenberg}	
	Every finite nilpotent group $G$ of class at most~$2$ is isomorphic to a subgroup of a non-degenerate Heisenberg group of the form \eqref{eq:matrixGroupH}
	for a suitable $\mu\from A\times B\to C$ depending on $G$. 
	Here the number of generators and exponents of $A,B,C$ are bounded by concrete functions of $G$ as in \eqref{eq:Droperties} and \eqref{eq:ExpProperties}.
\end{thmMain}
This statement is a direct consequence of the more technical \autoref{thm:embedToH}.
See \cite[Corollary~2.21]{Magidin} for a weaker conclusion in a much more general setup.
We remark here that both \autoref{thm:mainStructure} and \autoref{thm:mainHeisenberg} are true in the finitely generated setup, see the thesis of the author \cite[Theorems A,B]{phd}. 
Note that the number of isomorphism classes of groups of order $p^n$ is $p^{\frac{2}{27}n^3+\mathcal{O}(n^{8/3})}$ as $n\to\infty$ \cite[Theorem,~p.~153]{Sims} of which at least $p^{\frac{2}{27}n^3- \frac{12}{27}n^2}$ are nilpotent of class at most~$2$ \cite[Theorem~2.3]{Higman}.

\paragraph*{Topological applications}
Bounding the invariants in \autoref{thm:mainStructure} and \autoref{thm:mainHeisenberg} is essential for \autoref{thm:mainAction}. 
While  \autoref{thm:mainAction} shall be proved in a follow-up topological paper, we state it here and briefly explain its relevance to motivate \autoref{thm:mainStructure} and \autoref{thm:mainHeisenberg}, the main statements of the current paper.

Recall that a group $G$ is \emph{of rank at most $r$} if every subgroup $H$ of $G$ can be generated by at most $r$ elements.
\begin{cor}\label{thm:mainAction}
	For every natural number $r$, there exists 
	an algebraic variety $X_r$ and 
	a compact manifold $M_r$,
	such that every finite nilpotent group of class at most~$2$ and of rank at most~$r$  
	acts faithfully
	on $X_d$ via birational automorphisms and   
	on $M_d$ via diffeomorphisms.
\end{cor}
This statement about actions of varieties is sharp up to bounded extensions by \autoref{thm:birational2Nilpotent} and \autoref{rem:boundedRankBir} below, i.e. no larger set of finite groups can act simultaneously than the set of nilpotent groups of class at most~$2$ and of bounded rank. 
In fact, \autoref{thm:mainAction} demonstrates the sharpness of  \autoref{thm:birational2Nilpotent}.

\begin{defn}[{\cite[Definition~1]{guld2020finite2nilpotent}}]
	A group $G$ is nilpotently Jordan (of class at most $c$) if there is an integer $J_{G}$ such that every finite subgroup $F\leq G$ 
	contains a nilpotent subgroup $N\leq F$ (of class at most $c$) such that $|F:N|\leq J_G$. 
\end{defn}

\begin{thm}[{\cite[Theorem 2]{guld2020finite2nilpotent} based on \cite{prokhorov2016jordan}}]\label{thm:birational2Nilpotent}
	The birational automorphism group of any variety over a field of characteristic zero is nilpotently Jordan of class at most~$2$.	
\end{thm}

\begin{thm}[{\cite[Remark 6.9]{ProkhorovShramov2014} or  \cite[Theorem~15]{guld2019finiteDnilpotent} for details}]\label{rem:boundedRankBir}
	The rank of every finite group acting birationally on a variety $X$ over a field of characteristic zero is bounded by a function of $X$.
\end{thm}

The situation for manifolds is similar, but no bound on the nilpotency class is known 
apart from low dimensional cases \cite[Theorem 1.1]{Riera4dim}. 
The sharpness of \autoref{thm:mainAction} and \autoref{lem:boundedRankDiff} is an open question, cf. \cite[Question 1.5]{nilpotentJordanHomeo}.

\begin{thm}[{\cite[Theorem 1.3]{nilpotentJordanHomeo}}]\label{thm:homeoNilpotent}
	The homeomorphism group of every compact topological manifold is nilpotently Jordan.
\end{thm}

\begin{thm}[{\cite[Theorem~1.8]{CsikosMundetPyberSzabo}}]\label{lem:boundedRankDiff}
	The rank of every finite group acting continuously on a compact manifold $M$ is bounded by a function of $M$.
\end{thm}

\paragraph{Structure of the paper}
In \autoref{sec:subdirectProduct}, we prove \autoref{part:mainStructureSubdirect} of \autoref{thm:mainStructure} using induction and find the smallest number of factors needed for the subdirect product. 
In \autoref{sec:alterntingModules}, we turn our attention to nilpotent groups $G$ of class at most~$2$ with cyclic centre. 
The commutator map on $G$ induces an alternating bilinear map on the $\Z$-module $G/G'$. 
We prove \autoref{part:mainStructureCentral} of \autoref{thm:mainStructure} by finding a suitable generating set  imitating the Darboux basis of symplectic vector spaces. 
We define an additional isotropic complex structure on $G/\Center(G)$. 
In \autoref{sec:HeisenbergGroups}, this additional structure enables us to assign a Heisenberg group to $G$ using a method that is independent of the prime divisors of $|G|$. 
Finally, in \autoref{sec:HeisenbergEmbedding}, we discuss the general idea of modifying the previous construction by extending the centre so that actually $G$ embeds to the resulting Heisenberg group.
Then we prove \autoref{thm:mainHeisenberg} in two steps. First, we consider the special case when the group has cyclic centre and apply the method of the previous two sections. 
Second, we use the reduction of \autoref{sec:subdirectProduct} to handle the general case.

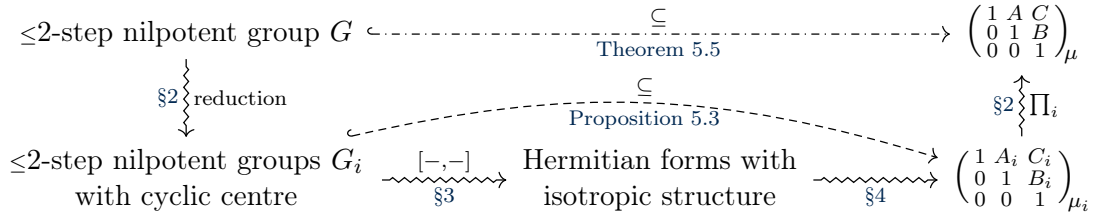
\[\begin{tikzcd}[column sep=4em,row sep=2em]
	\cbox{0.3\textwidth}{nilpotent group $G$ of class at most~$2$} 
		\ar[d,rightsquigarrow,"\text{reduction}","\text{\autoref{sec:subdirectProduct}}"'] 
		\ar[rr,dotted,inclusion,"\text{\autoref{thm:embedToH}}"']
	&
	& \cbox{0.3\textwidth}{$\matrixGroupH{A}{B}{C}{\mu}$} 
	\\
	\cbox{0.3\textwidth}{nilpotent groups $G_i$ of class at most~$2$ with cyclic centres} 	\ar[r,rightsquigarrow,"\text{\autoref{sec:alterntingModules}}"',"{[-,-]}"] 
		\ar[rr,mono,dashed,bend left=15, inclusion, "\text{\autoref{thm:maximalCbACyclicCentreEmbedsToHeisenberg}}"']
	& 
	\cbox{0.23\textwidth}{Hermitian forms with isotropic structure} 
		\ar[r,rightsquigarrow,"\text{\autoref{sec:HeisenbergGroups}}"']
	& 
	\cbox{0.3\textwidth}{$\matrixGroupH{A_i}{B_i}{C_i}{\mu_i}$} 
		\ar[u,rightsquigarrow,"\text{\autoref{sec:subdirectProduct}}","\prod_i"']
\end{tikzcd}\]
To help the reader, we demonstrate the details and the ideas of the construction above for the two non-abelian $p$-groups of order $p^3$ 
in a series of example running from  \autoref{sec:alterntingModules} to \autoref{sec:HeisenbergEmbedding}.

The present paper is the first of two compiled from the thesis of the author \cite{phd}.

\paragraph{Notation}
The \emph{cardinality} of a set $X$ is written as $|X|$. 
Following the convention of set theory, we abbreviate $\bigcap_{x\in X} x$ by $\bigcap X$.

\emph{Divisibility} is denoted by $n\divides k$, and the least common multiple of integers $n_1,\dots,n_k$ is written as $\lcm(n_1,\dots,n_k)$. (Note that the divisibility symbol is slightly taller than the one denoting the cardinality.) 

We apply maps \emph{from the left}. 
For $f\from X\to Y$ and a subset $X_0\subseteq X$, we write $f|_{X_0}\from X_0\to Y$ for its  \emph{restriction}, and $f(X_0)\leteq \{f(x_0):x_0\in X_0\}$ for its \emph{image}.  
Denote the \emph{direct product} of maps $f_i\from X_i\to Y_i$ by $f_1\times f_2\from X_1\times X_2\to Y_1\times Y_2$ which is defined by $(x_1,x_2)\mapsto (f_1(x_1),f_2(x_2))$.

Indicate a \emph{monomorphism} by $\arrowInline{mono}$, 
an \emph{epimorphism} by $\arrowInline{epi}$ and 
an isomorphism by $\arrowInline{iso}$.
We write $\arrowInline{identity}$ for the \emph{identity map} and 
$X\arrowInline{inclusion} Y$ for the \emph{natural inclusion map} for $X\subseteq Y$. 
In bigger diagrams, we use $\arrowInline{solid}$, 
$\arrowInline{dashed}$ and  
$\arrowInline{dotted}$ respectively to indicate whether the arrow appeared in the beginning, in the middle or in the end of the construction.

Let $G$ denote a group. 
We denote the \emph{identity element} of $G$ by $1$ (or sometimes by $0$ when $G$ is an additive abelian group). 
By abuse of notation, we also write $1$ (or $0$) for the \emph{trivial group}. 
We write $N\normal G$ to indicate a \emph{normal} subgroup $N$. 
For a positive integer $n$, we write $\Z_n\leteq \Z/n\Z$ for the cyclic group of order $n$ and abbreviate $k\Z\in\Z_n$ by $[k]$ if $n$ is understood.
For a subset $S\subseteq G$, $\generate{S}$ denotes the subgroup generated by $S$, 
and we write $\generate{g_1,\dots,g_n}\leteq \generate{\{g_1,\dots,g_n\}}$. 
We write $d(G)$ for the \emph{cardinality of the smallest generating set},  
and say that $G$ is \emph{$d$-generated} if $d=d(G)$.
We write $[-,-]\from G\times G\to G',(g,h)\mapsto[g,h]$ for the \emph{commutator map} where we use the convention $[g,h]\leteq g^{-1}h^{-1}gh$. 
Denote the \emph{commutator subgroup} (or \emph{derived subgroup}) of $G$ by $G'\leteq [G,G]$, 
the \emph{centre} by $\Center(G)$, 
and the \emph{exponent} by $\exp(G)$. 

\section{Subdirect product decomposition}\label{sec:subdirectProduct}

\begin{summary}
	The goal of this section is to prove \autoref{part:mainStructureSubdirect} of \autoref{thm:mainStructure} from \autopageref{thm:mainStructure} by passing to abelian groups. 
	To show the existence of such a subdirect product, we recursively take quotients by the invariant factors of the centre. 
	To attain the minimal number of factors, we consider intersections with the centre.
\end{summary}

In this section, we write $\CC$ for the class of groups with cyclic centre.
\begin{defn}\label{defn:CdecompositioninG}
	A \emph{$\CC$-decomposition in a group $G$} is a \emph{finite} set $D$ of normal subgroups of $G$ such that $G/N\in\CC$ for every $N\in D$ and $\bigcap D=1$. (Use the convention $\bigcap D = G$ if $D=\emptyset$.)
	Let $m_\CC(G)$ denote the minimal $|D|$ amongst all $\CC$-decomposition $D$ in $G$, or $\infty$ if no such decompositions exist.
\end{defn}

\begin{rem}\label{rem:associatedEmbedding}
	This is a reformulation of subdirect products, as the  \emph{associated (central) embedding}
	$\mu_D\from G\embeds G/D\leteq \prod_{N\in D}G/N,
		gK\mapsto(gN)_{N\in D}$ 
	makes $G$ a subdirect product of groups from $\CC$. 
\end{rem}

\begin{lem}\label{lem:mCFinite}
	There is a $\CC$-decomposition in every finite group $G$.
	Furthermore, $d(\Center(G))\leq m_\CC(G)$. 	
\end{lem}
\begin{proof}
	Let $l(G)$ be the maximal length of a strictly increasing subgroup series consisting of normal subgroups of $G$. Note that $l(G/N)<l(G)$ for any non-trivial normal subgroup $N$ of $G$ as a series $K_0/N<K_1/N<\dots<K_n/N$ of normal subgroups of $G/N$ induces $1<N<K_1<\dots <K_n$ in $G$.
	Write $\Center(G)=\prod_{i=1}^{d} C_i$ where $C_i$ are non-trivial cyclic groups. 
	If $d\leq 1$ (for example when $l(G)=0$), then $D=\{1\}$ is a $\CC$-decomposition in $G$. 
	Otherwise, by induction on $l(G)$, there are $\CC$-decompositions $D_i$ in $G/C_i$. 
	Lift $D_i$ to a set of normal subgroups $\bar D_i$ of $G$ containing $N$. i.e. $D_i=\{K/C_i:K\in\bar D_i\}$. 
	We claim that $D\leteq \bigcup_{i=1}^{d} \bar D_i$ is a $\CC$-decomposition in $G$. 
	Indeed, $D$ is a finite set of normal subgroups of $G$. For every $K\in D$, we have $K/C_i\in D_i$ for some $i$, hence $G/K\isom (G/C_i)/(K/C_i)\in\CC$ as $D_i$ is a $\CC$-decomposition in $G/C_i$. 
	Finally, note that 
	$\bigcap D=
	\bigcap_{i=1}^{d} \bigcap\bar D_i =
	\bigcap_{i=1}^{d} C_i =
	1$.

	For the second part, suppose $D$ is a $\CC$-decomposition in $G$. Then $\mu_D(\Center(G))\subseteq \Center(G/D)=\prod_{N\in D} \Center(G/N)$, so $d(\Center(G))\leq \sum_{N\in D} d(\Center(G/N))\leq |D|$ since $G/N$ has cyclic centre by assumption and using that $d$ is a monotone function on abelian groups.
\end{proof}

The next statement is motivated by an idea of Endre Szabó from a private communication.
\begin{lem}\label{lem:triviallyIntersectionInFiniteAbelian}
	Let $A$ be a finite abelian $p$-group, and let $X$ be a set of subgroups of $A$ such that $\bigcap X=1$. 
	Then there exists $Y\subseteq X$ with $|Y|\leq d(A)$ and  $\bigcap Y=1$. 
\end{lem}
\begin{proof}
	We prove this by induction on $d(A)$. 
	If $d(A)=0$, then $A$ is trivial, and $Y=\emptyset$ works by convention. 
	Otherwise, assume that $d(A)>0$. For any subgroup $K\leq A$, define $V(K)=\{g\in K:g^p=1\}$. Note that $V(K)$ can be considered as an $\F_p$-vector space of dimension $d(K)$.  
	Assume by contradiction that $V(K)=V(A)$ for all $K\in X$. Then $V(A)\subseteq K$, hence $V(A)\subseteq \bigcap X=0$, but this contradicts that $V(A)$ has positive dimension. 
	So we may pick $B\in X$ such that $d(B)<d(A)$. Now $X_B\leteq \{B\cap K:K\in X\}$ is a trivially intersecting set of subgroups of $B$, so by induction, there is $Y_B\subseteq X_B$ of size at most $d(B)$ with trivial intersection. 
	Lift back $Y_B$ to $Z_B\subseteq X$. Then $|Z_B|=|Y_B|$ and $Y_B=\{B\cap K:K\in Z_B\}$. 
	We show that $Y\leteq \{B\}\cup Z_B\subseteq X$ satisfies the claim. 
	Indeed, $|Y|\leq 1+|Y_B|\leq 1+d(B)\leq d(A)$ by construction, and 
	$\bigcap Y=
	B \cap \bigcap Z_B=
	\bigcap_{K\in Z} (B\cap K)=
	\bigcap Y_B=1$.
\end{proof}

\begin{lem}\label{lem:mCpGroup}
	For every finite $p$-group $P$, we have $m_\CC(P)=d(\Center(P))$.
\end{lem}
\begin{proof}
	There is a $\CC$-decomposition $D$ in $P$ by \autoref{lem:mCFinite}. 
	We claim the existence of a $\CC$-decomposition $S\subseteq D$ of size at most $d(\Center(P))$.  
	This then proves the statement as no smaller $\CC$-decomposition may exist by \autoref{lem:mCFinite}.
	
	Let $A\leteq \Center(P)$ and consider $X\leteq \{N\cap A:N\in D\}$, a trivially intersecting set of subgroups of the abelian group $A$. 
	Let $Y\subseteq X$ with $|Y|\leq d(A)$ and $\bigcap Y=1$ be given by \autoref{lem:triviallyIntersectionInFiniteAbelian}.
	Lift $Y$ back to $S\subseteq D$. 
	Then $1=\bigcap Y=\Center(P)\cap \bigcap S$, so we must have $\bigcap S=1$ since in a nilpotent group, there is no non-trivial normal subgroup intersecting the centre trivially. 
	Also, by construction, $|S|=|Y|\leq d(A)=d(\Center(P))$, so $S$ is indeed a $\CC$-decomposition of $P$ with the stated properties.
\end{proof}

\begin{lem}\label{lem:mCSylowMax}
	Let $G$ be a finite nilpotent group. 
	Then $m_\CC(G)=\max\{m_\CC(P):P\in \Syl(G)\}$ where $\Syl(G)=\bigcup_p \Syl_p(G)$ is the set of all Sylow subgroups of $G$.
\end{lem}
\begin{proof}
	Let $D$ be a $\CC$-decomposition in $G$ and $P\in\Syl(G)$. 
	We claim that $D_P\leteq \{N\cap P:N\in D\}$ is a $\CC$-decomposition in $P$. 
	Indeed, $P$ is a normal subgroup of $G$ because $G$ is finite nilpotent, so $N\cap P$ is a normal subgroup of $G$, hence of $P$. On the other hand $\bigcap N_P=P\cap \bigcap D=1$.
	This shows that $m_\CC(G)\geq m_\CC(P)$ for all $P\in\Syl(G)$.
	
	For the other direction, consider all prime divisors $p$ of $|G|$ and  
	let $D_p$ be a $\CC$-decomposition for the unique Sylow $p$-subgroup $G_p$ of $G$. 
	Let $\mathcal{D}$ be a partition of $\bigcup_p D_p$ of size $\max\{|D_p|:p\}$ such that $|S\cap D_p|\leq 1$ for every $S\in \mathcal{D}$ and $p$. 
	We claim that $D\leteq \{\prod_{N\in S}N:S\in \mathcal{D}\}$ is a $\CC$-decomposition in $G$.
	Indeed, as above, every $N_p\in D_p$ is normal in $G$, so every $N\in D$ is also a normal subgroup of $G$ being the product of such groups in a finite nilpotent group. 
	On the other hand, 
	$G_p\cap \bigcap D = 
	\bigcap \{G_p\cap \prod_{N\in S}N:S\in \mathcal{D}\} = 
	\bigcap \{K:K\in D_p\} = 
	\bigcap D_p=1$ using the assumption on the partition. 
	This shows that $m_\CC(G)\leq \max\{m_\CC(P):P\in\Syl(G)\}$.
\end{proof}

\begin{prop}\label{prop:subdirectProductFinite}
	$m_\CC(G)=d(\Center(G))$ for any finite nilpotent group $G$. 
\end{prop}
\begin{proof}
	Using 	\autoref{lem:mCSylowMax}, \autoref{lem:mCpGroup} and the Chinese remainder theorem, we get 
	\begin{align*}
		m_\CC(G) &= 
		\max\{m_\CC(P):P\in\Syl(G)\}=
		\max\{d(\Center(P)):P\in\Syl(G)\}\\&=
		\max\{d(Q):Q\in\Syl(\Center(G))\}=
		d(\Center(G))
	\end{align*} 
	after noting that $G=\prod_{P\in \Syl(G)} P$ implies $\Syl(\Center(G))=\{\Center(P):P\in\Syl(G)\}$.
\end{proof}

\begin{rem}
	Similar reasoning shows that the statement holds when $G$ is finitely generated, see \cite[\S3.1.2]{phd}.
\end{rem}

\begin{proof}[Proof of {\autoref{thm:mainStructure}/\autoref{part:mainStructureSubdirect}}]
	This is a straightforward consequence of \autoref{defn:CdecompositioninG}, 
	\autoref{rem:associatedEmbedding} and 
	\autoref{prop:subdirectProductFinite}.
\end{proof}

\section{Alternating modules}\label{sec:alterntingModules}
\begin{summary}
	In this section, we introduce alternating modules to generalise the notion of symplectic vector spaces: 
	the role of vector spaces is replaced by modules over principal ideal domains, and symplectic forms by alternating maps taking values in cyclic modules. 
	The key examples of such structures come from the abelianisation of nilpotent groups of class at most~$2$ together with the map induced by the commutator. 
	We show that, under some conditions, alternating modules  possess an analogue of the Darboux basis. 
	Using this, on one hand, we prove \autoref{part:mainStructureCentral} of \autoref{thm:mainStructure} from \autopageref{thm:mainStructure}, 
	on the other hand, we endow the alternating module with a non-canonical isotropic complex structure.
\end{summary}

We compare and contrast the following groups to illustrate the key ideas of the proof of \autoref{thm:mainHeisenberg}. These groups differ just enough to display the subtleties of the whole argument while keeping the complexity to a minimum.
This example is broken into several smaller parts spanning the rest of the paper. 
\begin{exmp}[$\extraSPecialD{n}$, $\extraSPecialQ{n}$: Definition]\label{exmp:DQdef}
	For any integer $n\geq 2$, define two $2$-generated nilpotent groups of class $2$ and of order $n^3$ by the following presentation.
	\begin{align*}
		\extraSPecialD{n}&\leteq \generate{\alpha,\beta,\gamma:
			\gamma=[\alpha,\beta],
			1=[\gamma,\alpha]=[\gamma,\beta], 
			\alpha^{n}=\beta^{n}=1,
			\gamma^n=1
		}\\
		\extraSPecialQ{n}&\leteq \generate{\alpha,\beta,\gamma:
			\gamma=[\alpha,\beta],
			1=[\gamma,\alpha]=[\gamma,\beta], 
			\alpha^{n}=\beta^{n}=\gamma,
			\gamma^n=1
		}
	\end{align*}
	This non-standard notation is motivated by the fact that $\extraSPecialD{2}\isom D_8$ (the dihedral group) and
	$\extraSPecialQ{2}\isom Q_8$ (the quaternion group).  
	For every prime $p$, every non-abelian group of order $p^3$ is isomorphic to one of $\extraSPecialD{p}$ and $\extraSPecialQ{p}$. These two (so-called extra-special) $p$-groups are the building blocks of \autoref{thm:specialPgroups}.
	Just like $D_8\not\isom Q_8$, we have $\extraSPecialD{n}\not\isom\extraSPecialQ{n}$ for every $n\geq 2$. 
	Indeed, a short computation shows that 
	$\exp(\extraSPecialD{n})=2n$ for even $n$ and 
	$\exp(\extraSPecialD{n})=n$ for odd $n$, whereas 
	$\exp(\extraSPecialQ{n})=n^2$ for every $n$. 
\end{exmp}

\subsection{Darboux generators and central product decomposition}

We start the main part of this section with an elementary statement about matrices.
\begin{lem}[`Alternating Smith' normal form]\label{lem:alternatingSmith}
	Let $R$ be a principal ideal domain, 
	$W\in R^{n\times n}$ be an alternating matrix (i.e. $W^\top = -W$ and has $0$'s at the main diagonal). 
	Then for $s=\frac{1}{2}\rank(W)\in\Z$, there exist elements $d_1\mid d_2\mid\dots\mid d_s\neq 0$ in $R$ (unique up to unit multiples) and $B\in \SL_n(R)$ such that 
	\begin{equation}
		B^\top W B = \diag\left(
		\begin{pmatrix}0 & d_1 \\ -d_1 & 0\end{pmatrix},
		\begin{pmatrix}0 & d_2 \\ -d_2 & 0\end{pmatrix},
		\dots,
		\begin{pmatrix}0 & d_s \\ -d_s & 0\end{pmatrix},
		0,\dots,0\right)
		.
		\label{eq:alternatingSmith}
	\end{equation}
\end{lem}

\begin{proof}
	
	The idea is similar to the standard proof of Smith normal form \cite[\S3.7]{jacobson1985basic}, but instead of focusing on the main diagonal, we consider the superdiagonal entries. At each step, we choose different pivots and apply the base change $W\mapsto X^\top W X$ (to respect the alternating property) to a series of well-chosen matrices $X\in\SL_n(R)$ until the stated form for $W=(w_{i,j})$ is obtained. 
	Once the existence is verified, the uniqueness statement follows from the fact that for each $1\leq k\leq n$, the ideal generated by the $k\times k$ minors is unchanged under these transformations. 
	See \cite[\S2.3]{phd} for details.
\end{proof}

A key notion of this section is the following analogue of symplectic vector spaces.
\begin{defn}\label{defn:alternatingModule}
	For a  principal ideal domain $R$, 
	we call $(M,\omega, C)$ an \emph{alternating $R$-module}, if 
	$M$ is a finitely generated $R$-module, 
	$C$ is a cyclic $R$-module, and 
	$\omega\from M\times M\to C$ is an $R$-bilinear map that is alternating, i.e. $\omega(m,m)=0$ for every $m\in M$. 
	We say that the alternating $R$-modules $\omega_i\from M_i\times M_i\to C_i$ (for $i\in\{1,2\}$) are \emph{isomorphic}, if there are isomorphisms $\lambda\from M_1\to M_2$ and $\kappa\from C_1\to C_2$ of $R$-modules such that $\kappa\circ\omega_1 = \omega_2\circ (\lambda\times \lambda)$ as $M\times M\to C_2$ maps.
	
	In the case above, for submodules $N,N_1,N_2\leq M$, introduce the following definitions.
	Let $\omega(N_1,N_2)\leq C$ be the submodule generated by $\{\omega(n_1,n_2):n_i\in N_i\}$. 
	Say $N_1$ and $N_2$ are \emph{orthogonal (with respect to $\omega$)}, written $N_1\perp N_2$, if $\omega(N_1,N_2)=0$. 
	Call $N$ \emph{isotropic} if $N\perp N$. 
	Call $N^\perp \leteq \{m\in M:\omega(m,N)=0\}$ the \emph{orthogonal submodule} (and note that $N$ and $N^\perp$ are not necessarily complementary).
	Call $\omega$ and $(M,\omega,C)$ \emph{non-degenerate} if $M^\perp=0$. 
\end{defn}

\begin{lem}[Darboux-generators]\label{lem:DarbouxGenerators}
	Let $(M,\omega,C)$ be an alternating $R$-module. 
	Then there exists a minimal $R$-module generating set $\mathcal{B}$ of $M$, and a subset $\{x_1,y_1,\dots,x_t,y_t\}\subseteq\mathcal{B}$
	such that 
	$\omega(M,M)=R\omega(x_1,y_1)\geq R \omega(x_2,y_2)\geq \dots \geq R \omega(x_t,y_t)\neq 0$ 
	and $\omega(b_1,b_2)=0$ for all other pairs $(b_1,b_2)\in \mathcal{B}^2$.
	
	The chain of submodules of $C$ above is independent of the choice of any such generating set. 
	More concretely, $t=\frac{1}{2}d(M/M^\perp)\in\Z$ and $M/M^\perp \isom \bigoplus_{i=1}^t (R\omega(x_i,y_i))^{\oplus 2}$.
	
\end{lem}

\begin{proof}
	Let $(M,\omega,C)$ be an alternating module.
	Pick a minimal $R$-module generating set $\{b_1,\dots,b_n\}$ of $M$, and let $c$ be a fixed generator of $\omega(M,M)$. 
	Pick $w_{i,j}\in R$ such that $\omega(b_i,b_j) = w_{i,j} c$. (Note that these are not necessarily unique, but $w_{i,j}+\ann_R(c)\in R/\ann_R(c)$ are.) 
	Without loss of generality, we may assume that $W\leteq (w_{i,j})_{i,j}\in R^{n\times n}$ is an alternating matrix.
	Let $B\in \SL_n(R)$ given by \autoref{lem:alternatingSmith}. 
	Since $B$ is an invertible square matrix, 
	$\mathcal{B}\leteq \{\sum_{j=1}^n B_{j,i} b_j:1\leq i\leq n\}=\{x_1,y_1,\dots,x_s,y_s,\dots\}$ is also a (minimal) generating set of $M$ in which $\omega$ can be expressed at the matrix from \eqref{eq:alternatingSmith}.
	Then the statement follows by setting the integer $t$ so that $d_ic\neq 0$ for $1\leq i\leq t$, and $d_jc=0$  for $t<j\leq s$.

	For the second part, let $x_i,y_i$ be as above. We claim that
	\begin{equation}
		\begin{tikzcd}[align]
		0\ar[r] & M^\perp \ar[r,inclusion] & M\ar[r,"\omega^\flat"] & \bigoplus_{i=1}^t (R\omega(x_i,y_i))^{\oplus 2} \ar[r] & 0 \\
		& & m \ar[r,mapsto] & (\omega(m,y_i),\omega(x_i,m))_{i=1}^t
		\end{tikzcd}
		\label{diag:omegaFlat}
	\end{equation}
	is a short exact sequence of $R$-modules. 
	Indeed, $\omega^\flat$ is well-defined using the bilinearity of $\omega$ and the orthogonality elements of $\mathcal{B}$.
	Note that $\ker(\omega^\flat) = 
	\{m\in M:\forall i\;\omega(x_i,m)=\omega(m,y_i)\}=
	\{m\in M:\forall b\in \mathcal{B}\;\omega(b,m)=0\}=
	M^\perp$.
	On the other hand, using orthogonality once more, 
	$\omega^\flat(\sum_{i=1}^t r_ix_i + s_iy_i)=(r_i,y_i)_i$ for arbitrary $r_i,s_i\in R$, 
	thus $\omega^\flat$ is surjective.	
	Hence $M/M^\perp\isom \bigoplus_{i=1}^t (R\omega(x_i,y_i))^{\oplus 2}$, so $t=\frac{1}{2}d(M/M^\perp)$ and the isomorphism class of the $R$-modules $R\omega(x_i,y_i)$ are invariant by the structure theorem of finitely generated modules over a principal ideal domain.
	If $\ann_R(c)\neq 0$, then $\omega(M,M)$ is a cyclic torsion $R$-module, so its isomorphic submodules are necessarily equal, meaning that the submodules $R\omega(x_i,y_i)\leq \omega(M,M)$ are themselves invariant.
	Otherwise, $\omega(M,M)$ is a free $R$-module of rank $1$ generated by $c$. Thus $\omega(x_i,y_i)=d_ic$ for some unique $d_i(c)\in R$. By \autoref{lem:alternatingSmith}, $Rd_i$ is independent of the choice of the generators of $M$. Hence $Rd_ic=R\omega(x_i,y_i)$ may depend only on $c$, but the right-hand side does not.
\end{proof}

Our main source of alternating $\Z$-modules is the following.
\begin{defn}\label{defn:centralByAbelian}
	A short exact sequence 
	$\epsilon:1\to C\xrightarrow{\iota} G\xrightarrow{\pi} M\to 1$
	of groups is called a \emph{\centralByAbelian{} extension}, if $\iota(C)\subseteq\Center(G)$ and $M$ is abelian. 
	This extension $\epsilon$ is \emph{non-degenerate}, if $\iota(C)=Z(G)$.
\end{defn}

\begin{lem}[The alternating functor $\alternatingFunctor$]\label{lem:alterntingFunctor}
	Every \centralByAbelian{} extension groups 
	\[\begin{tikzcd}[label]
		\epsilon\ar[:] 
		& 1 \ar[r] 
		& C\ar[r,mono,"\iota"] 
		& G\ar[r,epi,"\pi"]
		& M\ar[r] 
		& 1
	\end{tikzcd}\]
	induces an alternating $\Z$-bilinear map 
	\[
		\omega\from M\times M\to C, \qquad 
		(m_1,m_2)\mapsto \iota^{-1}([g_1,g_2]) 
	\]
	defined by arbitrary $g_i\in\pi^{-1}(m_i)$. 
	
	In particular, when $M$ is finitely generated and $C$ is cyclic, then $\alternatingFunctor(\epsilon)\leteq (M,\omega,C)$ is an alternating $\Z$-module.  
\end{lem}

\begin{proof}
	We consider $M$ and $C$ as $\Z$-modules. 
	First note that $G'\subseteq\iota(C)\subseteq\Center(G)$, so $G$ is necessarily nilpotent of class at most~$2$. 
	Then the general commutator identities 
	$[g_1g_2,h]=[g_1,h][g_1,h,g_2][g_2,h]$ and 
	$[g,h_1h_2]=[g,h_2][g,h_1][g,h_1,h_2]$
	imply that $[-,-]:G\times G\to G'$ is a group morphism in both coordinates.
	
	Next, we check that $\omega$ is well defined. Pick $g_i,g_i'\in \pi^{-1}(m_i)$ for $i\in\{1,2\}$. Then $g_i^{-1}g_i'\in \ker(\pi)=\im(\iota)$, so there is $c_i\in C$ with $\iota(c_i)=g_i^{-1}g_i'$. 
	Then 
	$[g_1',g_2']=
	[g_1\iota(c_1),g_2\iota(c_2)]=
	[g_1,g_2][g_1,\iota(c_2)][\iota(c_1),g_2][\iota(c_1),\iota(c_2)]=
	[g_1,g_2]$ 
	by above as $\iota(C)\subseteq Z(G)$. 
	Finally, $G'\subseteq \iota(C)$ implies that we can apply $\iota^{-1}$ to this element.
	
	$\Z$-bilinearity of $\omega$ follows directly from the previously mentioned fact. The alternating property follows as every group element commutes with itself.
	\end{proof}

\begin{exmp}[$\extraSPecialD{n}$, $\extraSPecialQ{n}$: Alternating maps -- continuing \autoref{exmp:DQdef}]\label{exmp:DQomega}
	Let $G\in\{\extraSPecialD{n}, \extraSPecialQ{n}\}$. 
	Recall that $G$ is generated by $\{\alpha,\beta,\gamma\}$ and note that $\Center(G)=G'$ is cyclic of order $n$ generated by $\gamma$. 
	\autoref{lem:alterntingFunctor} applied to the non-degenerate \centralByAbelian{} extension 	
	\[\begin{tikzcd}[label]
		\maxCBAfunctor(G) \ar[:] & 
		1\ar[r] & 
		\Center(G)\ar[r,inclusion] & 
		G\ar[r,epi,"\maximalCBAProj"]  & 
		G/\Center(G)\ar[r]  & 1
	\end{tikzcd}\]
	gives a non-degenerate alternating $\Z$-module $(M,\omega_G, C)$ where $M\leteq G/\Center(G)$ and $C\leteq\Center(G)$. 
	Note that 
	$M=\generate{\coset{\alpha} , \coset{\beta}}$ where $\coset{ g}\leteq \maximalCBAProj(g)$ for $g\in G$. 
	Thus $\omega_G\from M\times M \to C$ is given by $(\coset{\alpha}, \coset{\beta})\mapsto \gamma$. 
	This shows that even though $\extraSPecialD{n}$ and $\extraSPecialQ{n}$ are not isomorphic, the resulting $\omega_{\extraSPecialD{n}}$ and $\omega_{\extraSPecialQ{n}}$ are isomorphic (in the sense of \autoref{defn:alternatingModule}).
\end{exmp}

\begin{rem}\label{rem:dictionary}
	\autoref{lem:alterntingFunctor} gives the following dictionary between subgroups $H$, $H_i$ of $G$  and submodules $\pi(H)$, $\pi(H_i)$ of $M$.
	\begin{itemize}
		\item The commutator map corresponds to $\omega$: $[g,g']=\iota\circ\omega(\pi(g),\pi(g'))$.
		
		\item Commuting subgroups correspond to orthogonal submodules: $[H_1,H_2]=1$ if and only if $\pi(H_1)\perp \pi(H_2)$. 
		(In particular, $H$ is abelian if and only if $\pi(H)$ is isotropic.)
		
		\item The centraliser of a subgroup corresponds to the orthogonal submodule: we have $\pi(C_G(H))= (\pi(H))^\perp$.  
		(In particular, $\pi(\Center(G))=M^\perp$, hence the non-degeneracy of the \centralByAbelian{} extension and that of the alternating module coincide.)
	\end{itemize}
\end{rem}

The dictionary can be extended to Darboux-generators as the following generalisation of \autoref{thm:specialPgroups} and \cite[Theorem~2.1]{brady_bryce_cossey_1969} shows. 
\begin{thm}[Central product decomposition]\label{cor:decompositionCyclicDeriverGroup}
	Let $G$ be a finite nilpotent group of class at most~$2$ with cyclic commutator subgroup $G'$. 
	Then it contains pairwise commuting subgroups $A$ and $E_1,\dots,E_t$ such that 
	$G=AE_1\dots E_t$ (a central product) where 
	$A\leq \Center(G)$, 
	$E_i$ are $2$-generated and of class exactly $2$, 
	$d(G)=d(A)+2t$ and
	$G'=E_1'\supsetneq E_2'\supsetneq \dots \supsetneq E_t'\neq 1$. 
	
	For any such subgroups $E_1,\dots,E_t$, the integer $t=\frac{1}{2} d(G/\Center(G))$ and $E_i'\subseteq G'$ are invariants of $G$ given by $G/\Center(G)\isom \prod_{i=1}^t E_i'^2$.
\end{thm}

\begin{proof}
	Let $(M,\omega,C)=\alternatingFunctor(\epsilon:1\to G'\xrightarrow{\subseteq} G\xrightarrow{\pi} G/G'\to 1)$ 
	be the alternating $\Z$-module given by \autoref{lem:alterntingFunctor}. 
	Let $\mathcal{B}=\{x_1,y_1,\dots,x_t,y_t,o_1,\dots,o_k\}$ be a minimal generating set of $M=G/G'$ as in \autoref{lem:DarbouxGenerators}.
	For every $b\in \mathcal{B}$, fix an arbitrary lift $\cosetRepresentative{b}\in\pi^{-1}(b)\subseteq G$, and set $\cosetRepresentative{\mathcal{B}} \leteq \{\cosetRepresentative{g}:g\in \mathcal{B}\}$. 
	We show that the subgroups 
	$E_i\leteq \generate{\cosetRepresentative{x_i},\cosetRepresentative{ y_i}}\leq G$ and $A\leteq \generate{\cosetRepresentative{o_1},\dots,\cosetRepresentative{o_k}}\leq G$  satisfy the statement.
	Indeed, $A\subseteq \pi^{-1}(M^\perp)=\Center(G)$ using \autoref{rem:dictionary}. 
	Moreover, 
	$[E_i,E_j]=1$
	if and only if 
	$(\Z x_i+\Z y_i)\perp (\Z x_j+\Z y_j)$
	if and only if 
	$i\neq j$ 
	by \autoref{rem:dictionary} and \autoref{lem:DarbouxGenerators}. 
	By \autoref{rem:dictionary}, 
	$G'=[G,G]=\omega(M,M)$ and 
	$E_i'=[E_i,E_i]=\omega(\Z x_i+\Z y_i,\Z x_i+\Z y_i)=\Z\omega(x_i,y_i)$. 
	Then all parts about the derived subgroups follow from \autoref{lem:DarbouxGenerators}.
	In particular, $G'=E_1'=\Z\omega(x_i,y_i)=\generate{[\cosetRepresentative{ x_1}, \cosetRepresentative{y_1}]}$, 
	so considering the short exact sequence $\epsilon:1\to \generate{[\cosetRepresentative{ x_1}, \cosetRepresentative{y_1}]}\to G\to \generate{\mathcal{B}}\to 1$, we see that 
	$G
	=\generate{\{[\cosetRepresentative{x_1}, \cosetRepresentative{y_1}]\}\cup \cosetRepresentative{\mathcal{B}}}
	= \generate{\cosetRepresentative{\mathcal{B}}}
	=AE_1\dots E_t
	$.
	So
	$d(G)\leq
	d(A)+\sum_{i=1}^t d(E_i)\leq 
	|\cosetRepresentative{\mathcal{B}}|
	=|\mathcal{B}|=d(M)\leq d(G)$. 
	This forces equality everywhere, so $d(E_i)=2$ and $2t+d(A)=d(G)$. 
	Finally \autoref{rem:dictionary} shows that $M^\perp=\pi(\Center(G))=\Center(G)/G'$, 
	so $M/M^\perp = (G/G')/(\Center(G)/G')\isom G/\Center(G)$.
\end{proof}

\begin{rem}\label{rem:centralProductDecompositionNotUnique}
	The isomorphism classes of the subgroups $E_1,\dots,E_t$ are not unique. 
	For example, the external central product $D_8\cprod Q_8$ of the dihedral and quaternion group is isomorphic to $Q_8\cprod Q_8$. 
	More generally, for a different prime $p>2$ and a general positive integer $t$, we can consider the classical decomposition of extra-special $p$-groups from \autoref{thm:specialPgroups}, cf. \cite[Theorem~4.18]{Suzuki2}.
	Let $\extraSPecialD{p}$ be the non-abelian group of order $p^3$ of exponent $p$, 
	and let $\extraSPecialQ{p}$ be the non-abelian group of order $p^3$ of exponent $p^2$, cf. \autoref{exmp:DQdef}. 
	Then the isomorphism class of the central product of $s$ copies of $\extraSPecialQ{p}$ and $t-s$ copies of $\extraSPecialD{p}$ is independent of the choice of $s$ provided $1\leq s\leq t$. 
\end{rem}

\begin{proof}[Proof of {\autoref{thm:mainStructure}/\autoref{part:mainStructureCentral}}]
	This is a special case of \autoref{cor:decompositionCyclicDeriverGroup}.
\end{proof}

\subsection{Complex structure}

If an alternating module is non-degenerate, we can endow it with additional structures. 

\begin{defn}
	For a commutative ring $Q$, define a ring $Q[i] \leteq Q[x]/(x^2+1)$ with $i\leteq x+(x^2+1)\in Q[i]$ and 
	define the maps
	$\sigma\from Q[i]\to Q[i],q+iq'\mapsto q-iq'$ (\emph{conjugation}) 
	and $\Im\from Q[i]\to Q,q+iq\mapsto q'$ (the \emph{imaginary part}). 
	For a $Q[i]$-module $M$, we call a map $h\from M\times M\to Q[i]$ a \emph{Hermitian form on $M$ over $Q[i]$} if $h$ is $Q[i]$-linear in the first argument and $h$ is $\sigma$-conjugate symmetric (i.e. $h(m,m')=\sigma(h(m',m))$).
\end{defn}

The next statement is essential for \autoref{thm:mainAction}.
\begin{prop}[Isotropic complex structure]\label{lem:complexStructure}
	Let $(M,\omega,C)$ be a non-degenerate alternating $R$-module. 
	Set $Q\leteq R/\ann_R(\omega(M,M))$, and 
	let $\phi\from Q\embeds C$ be a monomorphism of $R$-modules such that $\omega(M,M)\subseteq \phi(Q)$.
	Then $M$ can be endowed with a $Q[i]$-module structure together with a Hermitian form $h$ making the following diagram commute. 
	\[\begin{tikzcd}
		M\times M \ar[r,"\omega"] \ar[d,"h","\exists"',dashed]
		& C
		\\
		Q[i] \ar[r,"\Im",epi] 
		& Q \ar[u,"\phi"',mono]
	\end{tikzcd}\]
	Furthermore, $M$ has a non-canonical isotropic $Q$-structure, i.e. an isotropic $Q$-submodule $M_Q$ of $M$ such that $M=M_Q\oplus iM_{Q}$ (as $Q$-modules). 
	Finally, there is $\alpha\in M_Q$ such that $\Im (h(\alpha,i\alpha))$ is a $Q$-module generator of $Q$. 
\end{prop}	
\begin{rem}
	The $Q[i]$-module structure of $M$ from \autoref{lem:complexStructure} is non-canonical, but is compatible with $R\surjects Q\embeds Q[i]$.
	The $Q$-module $iM_Q$ is automatically isotropic as $\omega(ia,ia')=\phi(\Im(h(ia,ia')))=\phi(\Im(h(a,a')))=\omega(a,a')=0$.
	
	If $C$ is finite, then by order considerations, the condition $\omega(M,M)\subseteq \phi(Q)$ is automatically satisfied.
\end{rem}

\begin{rem}\label{rem:complexStructuresIsomorphic}
	While the structures themselves from \autoref{lem:complexStructure} depend on the choice of the generators of $M$, their isomorphism classes do not. 
	More concretely, let $M'$ be an arbitrary $Q[i]$-module structure on the $M$ together with 
	a Hermitian form $h'$ and an isotropic $Q$-structure 
	$M' = M'_Q\oplus iM'_Q$ as at the statement. 
	Then \autoref{lem:alternatingSmith} implies the existence of a $Q[i]$-module isomorphism $f\from M\to M'$ such that 
	$f(M_Q)=M'_Q$ and 
	$h'\circ(f\times f) = h$ (hence $\omega\circ(f\times f)=\omega)$.
\end{rem}

\begin{proof}[Proof of \autoref{lem:complexStructure}]
	First, we claim that $\ann_R(\omega(M,M))\subseteq \ann_R(M)$.
	Indeed,  for arbitrary $r\in \ann_R(\omega(M,M))$ and $m\in M$, 
	$\omega(rm,m')=r\omega(m,m')=0$ for every $m'\in M$, hence 
	$rm\in M^\perp=0$. 
	This then means that $M$ can be naturally considered as a $Q$-module.
	To define the $Q[i]$-module structure, let $\mathcal{B}=\{x_1,y_1,\dots,x_t,y_t,o_1,\dots,o_k\}$ be a Darboux-generating set as in \autoref{lem:DarbouxGenerators}. 	
	The isomorphism $\omega^\flat:M\isom \bigoplus_{j=1}^t (R\omega(x_j,y_j))^{\oplus 2}$ from \eqref{diag:omegaFlat} shows that $k=0$ and 
	$M=\bigoplus_{j=1}^t (Rx_i\oplus Ry_i)=\bigoplus_{j=1}^t (Qx_i\oplus Qy_i)$. 
	In particular, $M_Q\leteq \bigoplus_{j=1}^t Qx_i$ and $M_{iQ}\leteq \bigoplus_{j=1}^t Qy_i$ are (non-canonical) isotropic submodules of $M$ giving a $Q$-module decomposition $M=A\oplus B$.
	Define a $Q$-module automorphism $\iota_j$ of $(Q\omega(x_j,y_j))^{\oplus 2}$ by $\iota_j\from(n,n')\mapsto (-n',n)$. Pulling back $\bigoplus_{j=1}^t \iota_j$ along $\omega^\flat$ gives a non-canonical automorphism $\iota$ of $M$ 
	such that $\iota(x_j)=y_j$ and $\iota(y_j)=-x_j$. 
	Thus $\iota\circ \iota=-\id_M$, moreover 
	$\iota(M_Q)=M_{iQ}$ and $\iota(M_{iQ})=M_Q$. 
	Thus defining $(q+iq')\cdot m\leteq qm+\iota(q'm)$ gives the $Q[i]$-module structure in which $M_{iQ}=iM_Q$.

	By assumption, there is a $Q$-bilinear map $\omega_Q\from M\times M\to Q$ such that $\omega = \phi\circ \omega_Q$. 
	We claim that 
	\begin{align}
		h\from M\times M\to Q[i],\quad
		(m,m')\mapsto \omega_Q(im,m')+i\omega_Q(m,m')
		\label{eq:HermitianFromAlternating}
	\end{align}
	is the Hermitian form with the stated properties. 
	Indeed, $Q$-linearity in the first argument is inherited from $\omega$, and 
	\[h(im,m')
	=\omega_Q(-m,m')+i\omega_Q(im,m')
	=i(\omega_Q(im,m')+i\omega_Q(m,m'))
	=ih(m,m')
	\] then implies $Q[i]$-linearity.
	For conjugate symmetry, first note that $\omega(ix_j,iy_j)=\omega(y_j,-x_j)=\omega(x_j,y_j)$ using the alternating property, 
	whereas for all other pairs $(b_1,b_2)\in\mathcal{B}^2$, we also have $\omega(ib_1,ib_2)=0=\omega(b_1,b_2)$.
	Hence the $Q$-bilinearity of $\omega_Q$ implies 
	$\omega(im,im')=\omega(m,m')$ for every $m,m'\in M$.
	This together with the alternating property of $\omega_Q$ gives 
	$\omega_Q(im,m')
	=-\omega_Q(m',im)
	=\omega_Q(i^2m',im)
	=\omega_Q(im',m)
	$, 
	thus 
	\[h(m,m')
	=\omega_Q(im,m')+i\omega_Q(m,m')
	=\omega_Q(im',m)-i\omega_Q(m',m)
	=\sigma(h(m',m))
	.\]
	Finally, by construction, we can take $\alpha=x_1$.
\end{proof}

\begin{rem}\label{rem:omega-mu equivalence}
	The Hermitian form and the alternating map determine each other uniquely via 
	\eqref{eq:HermitianFromAlternating} and $\Im\circ h = \omega_Q$.
	Furthermore, given the isotropic $Q$-structure, these maps are determined by the restriction $\mu\from M_Q\times iM_Q\to C, (a,b)\mapsto \omega(a,b)$. 
	Indeed, note that $\omega(a+ib,a'+ib')= \mu(a,ib')-\mu(a',ib)$ using the bilinearity of $\omega$ and the fact that $M_Q$ is isotropic.
\end{rem}

\begin{exmp}[$\extraSPecialD{n}$, $\extraSPecialQ{n}$: Hermitian forms -- continuing \autoref{exmp:DQomega}]\label{exmp:DQhermitian}
	Let $R=\Z$ and consider the ring $Q\leteq \Z/(n)$. 
	Note that $M=G/\Center(G)$ is a $Q$-module generated by $\{\coset{\alpha},\coset{\beta}\}$. 
	In fact, $M_Q\leteq Q\coset{\alpha}$ is an isotropic $Q$-submodule such that $M=Q \coset{\alpha} \oplus Q \coset{\beta}$.
	The $Q[i]$-structure on $M$ is defined by $i\cdot (a\coset{\alpha} +b\coset{\beta})\leteq -b\coset{\alpha}+a\coset{\beta}$. 
	Let $\phi\from Q\to C=\generate{\gamma}$ be an isomorphism of $\Z$-modules such that $[1]\mapsto \gamma$. 
	Then the Hermitian form $h\from M\times M\to Q[i]$ of \autoref{lem:complexStructure} is given by $(a\coset{\alpha}+b\coset{\beta}, a'\coset{\alpha}+b'\coset{\beta})\mapsto -(a+ib)(a'-ib')=-(aa'+bb')+i(ab'-a'b)$. 
	Note that the $Q[i]$-structure is induced by the isomorphism $M\to Q[i],a\coset{\alpha}+b\coset{\beta}\mapsto a+bi$, and that 
	$h$ corresponds to $-1$ times the usual Hermitian form on $Q[i]$.
\end{exmp}

\section{Heisenberg groups}\label{sec:HeisenbergGroups}

\label{Subsections?}

\begin{summary}
	In this section, we associate a (polarised) Heisenberg group to every $\Z$-bilinear map, in particular to alternating modules, or to finite nilpotent groups $G$ of class at most $2$ with cyclic centre. 
\end{summary}

\begin{defn}[Heisenberg group]\label{def:H}
	Let $A$, $B$ and $C$ be $\Z$-modules and 
	$\mu\from A\times B\to C$ a (not necessarily surjective) $\Z$-bilinear map. 
	We call $\mu$ \emph{non-degenerate}, if $\mu(a,B)=0$ implies $a=0$ and $\mu(A,b)=0$ implies $b=0$. 
	Define the associated \emph{Heisenberg group} as
	$\HH(\mu)\leteq A\ltimes_\phi(B\times C)$
	where $\phi\from A\to \Aut(B\times C),a\mapsto((b,c)\mapsto (b,\mu(a,b)+c)$. 
	Call $\HH(\mu)$ \emph{non-degenerate} if $\Center(\HH(\mu))=\{(0,0,c):c\in C\}$.
	Define a \centralByAbelian{} extension 
	\[\begin{tikzcd}[label]
		\HeisenbergFunctor(\mu) \ar[:] 
		& 1\ar[r]
		& C \ar[r,"{\HeisenbergMono[\mu]}",mono] 
		& \HH(\mu) \ar[r,"{\HeisenbergEpi[\mu]}",epi]
		& A\times B \ar[r] 
		& 1
	\end{tikzcd}\]
	by 
	$\HeisenbergMono\leteq \HeisenbergMono[\mu]:c\mapsto(0,0,c)$ and $\HeisenbergEpi\leteq \HeisenbergEpi[\mu]:(a,b,c)\mapsto (a,b)$. 
\end{defn}

\begin{rem}\label{rem:Hdef}
	More explicitly, the group structure on $\HH(\mu)$ is given by  
	\[(a,b,c)*(a',b',c')  = (a+a', b+b',c + \mu(a,b')+c')\] with 
	$(0,0,0)$ being the identity and 
	$(a,b,c)^{-1} =(-a,-b,\mu(a,b)-c)$ the inverse. 
	So formally \[\HH(\mu)\isom\matrixGroupH{A}{B}{C}{\mu}\] with matrix multiplication induced by $\mu$. 
	In particular, we have  $[(a,b,c),(a',b',c')]=(0,0,\mu(a,b')-\mu(a',b))$, i.e. the commutator coincides with $\HeisenbergMono[\mu]\circ\omega\circ(\HeisenbergEpi[\mu]\times \HeisenbergEpi[\mu])$ using the notation of \autoref{rem:omega-mu equivalence}.
	Note that $\Center(\HH(\mu)) = \{(a,b,c):\mu(a,B)=\mu(A,b)=0\}\supseteq \HeisenbergMono[\mu](C)$. 
	The notion of non-degeneracy for $\mu$, $\omega$, $\HH(\mu)$ and $\HeisenbergFunctor(\mu)$ all coincide.	
\end{rem}

\begin{exmp}[$\extraSPecialD{n}$, $\extraSPecialQ{n}$: isomorphism to Heisenberg groups -- continuing \autoref{exmp:DQdef}]\label{exmp:DHQnotH}
	We claim that $\extraSPecialD{n}$ is isomorphic to a Heisenberg group, namely to the $3\times 3$ upper triangular matrices with entries from $\Z/(n)$. 
	Indeed, let the bilinear map $\mu_n\from \Z_n\times \Z_n\to \Z_n$ be defined by multiplication: $([a],[b])\mapsto [a\cdot b]$. 
	Then 
	\begin{equation}
		\label{eq:DisoH}
		\begin{aligned}
			\extraSPecialD{n}&\to\HH(\mu_n\from \Z_n\times \Z_n\to \Z_n)\\
		\beta^b\gamma^c\alpha^a&\mapsto ([a], [b], [c])
		\end{aligned}
	\end{equation} 
	is an isomorphism. 
	
	On the other hand, $\extraSPecialQ{n}$ is not isomorphic to any Heisenberg group (demonstrating that the embedding of \autoref{thm:mainHeisenberg} cannot be replaced by an isomorphism in general).  
	Indeed, seeking a contradiction, assume that $\extraSPecialQ{n}$ is isomorphic to some (necessarily non-degenerate) Heisenberg group $H=\HH(\mu\from A\times B\to C)$. 
	Since $\Center(\extraSPecialQ{n})=\extraSPecialQ{n}'=\generate{\gamma}$, we must have $\Center(H)=H'=\{(0,0,c):c\in C\}\isom C$ by \autoref{rem:Hdef}, 
	thus $C\isom \Z_n$.  
	Then $A\times B\isom H/H'\isom \extraSPecialQ{n}/\extraSPecialQ{n}' = \generate{\alpha \extraSPecialQ{n}'}\times  \generate{\beta \extraSPecialQ{n}'}\isom \Z_n\times\Z_n$, 
	which in turn implies that $A\isom B\isom \Z_n$ 
	since the triviality of $A$ or $B$ would mean that $H\isom\extraSPecialQ{n}$ is abelian. 
	If $A=\generate{a}$ and $B=\generate{b}$, then $C=\generate{\mu(a,b)}$ by above. 
	Hence by picking isomorphisms $A\to\Z_c,a \mapsto [1]$, 
	$B\to\Z_c,b \mapsto [1]$ and 
	$C\to\Z_c,\mu(a,b) \mapsto [1]$, 
	we see that $H\isom \HH(\mu_n)\isom\extraSPecialD{n}$. 
	But this cannot happen, as $\extraSPecialD{n}\not\isom\extraSPecialQ{n}$ using  \autoref{exmp:DQdef}. 
\end{exmp}

\begin{note}[Polarised Heisenberg group of alternating modules]\label{rem:HeisenbergOfAlternetingModule}
	Let $(M,\omega,C)$ be a non-degenerate alternating $R$-module. 	%
	Apply \autoref{lem:complexStructure}, 
	set $A\leteq M_Q$ and $B\leteq iM_Q$ considered as $\Z$-modules.  
	Now $M=A\oplus B$. 
	Then the restriction $\mu\from A\times B\to C$ as in \autoref{rem:omega-mu equivalence} is non-degenerate and produces the non-degenerate \centralByAbelian{} extension $\HeisenbergFunctor(\mu)$ of \autoref{def:H}. 
	In the symplectic vector space analogy, the isotropic submodule $A\leq M$ is called a polarisation, so we call $\HH(\mu)$ the \emph{polarised Heisenberg group} associated to the polarisation $A$ if the alternating $R$-module.
	
	Note that while this construction depends on the choice of the polarisation, 
	the isomorphism class of $\HeisenbergFunctor(\mu)$ does not, because $f$ from \autoref{rem:complexStructuresIsomorphic} induces an isomorphism of short exact sequences, cf. \cite[\S4.1]{phd}. 
	In particular, the isomorphism class of $\HH(\mu)$ is invariant (under the choice of the polarisation) and we call this isomorphism class the \emph{Heisenberg group associated} to the alternating module. 

	Note that $(M,\omega,C)$, $\mu$, $\HH(\mu)$ and $\HeisenbergFunctor(\mu)$  basically decode the same information as they mutually determine each other.
\end{note}
\begin{rem}[Canonical Heisenberg group]
	Consider the setup of  \autoref{rem:HeisenbergOfAlternetingModule}. 
	We wish to define $\HH(\mu)$ in a canonical way, i.e. we want to find a group $H_\omega$ isomorphic to $\HH(\mu)$ that depends only on $\omega$ and not on the choice of the polarisation. 
	For this, pick $x,y\in R$ whose values are to be determined later. 
	Being motivated by the standard construction for symplectic vector spaces, define a group on the set $H_\omega\leteq M\times C$ where the binary operation is given by $(m,c)\cdot (m',c')\leteq  (m+m',c+x\omega(m,m')+c')$.  
	A short computation shows that the map 
	\[\phi\from \HH(\mu\from A\times B\to C)\to H_\omega,\quad
	(a,b,c)\mapsto (a+b, y\omega(a,b) + c)\]
	is a group isomorphism  if and only if $(1-x+y)\omega(M,M)= (x+y)\omega(M,M)=0\in C$. 
	Since $\omega(M,M)$ is cyclic, we may write $\omega(M,M)\isom R/(d)$, where $d\in R$ is a generator of the annihilator of the $R$-module $\omega(M,M)$.  
	Hence the required $x,y\in R$ exist making $\phi$ an isomorphism if and only if $1\equiv 2x\pmod{d}$ has a solution in $x$, i.e. when $d\in R$ and $2\in R$ are relative prime. 
		
	In particular, if $M$ is finite, then the canonical construction above works if $\omega(M,M)$ is of odd order. 
	In fact, this is a necessary condition for $\HH(\mu)\isom H_\omega$ in the finite case, since  $\HH(\mu)'=\{(0,0,c):c\in \omega(M,M)\}\isom \omega(M,M)$ by \autoref{rem:Hdef} and 
	$H_\omega'=\{(0,2xc):c\in\omega(M,M)\}\isom 2x\omega(M,M)$ by direct computation. 
	Note that $\omega(M,M)$ having odd order is equivalent to $M$ having odd order by the non-degeneracy of $\omega$.

	In this paper, we shall use the polarised construction of \autoref{rem:HeisenbergOfAlternetingModule} instead of the canonical one to treat all cases uniformly --  whether the group has $2$-torsion or not. Accordingly, by a  Heisenberg group we mean the polarised construction: \autoref{def:H} and \autoref{rem:HeisenbergOfAlternetingModule}.
\end{rem}

We can associate a Heisenberg group to any nilpotent group of class at most~$2$ with cyclic centre as follows.
\begin{note}\label{exmp:mu_G}
	Let $G$ be a finite nilpotent group of class at most~$2$ with cyclic centre. 
	Apply \autoref{lem:alterntingFunctor} to the non-degenerate \centralByAbelian{} extension 
	\[\begin{tikzcd}[label]
		\maxCBAfunctor(G) \ar[:] & 
		1\ar[r] & 
		\Center(G)\ar[r,inclusion] & 
		G\ar[r,epi,"\maximalCBAProj"]  & 
		G/\Center(G)\ar[r]  & 1
	\end{tikzcd}\]
	to obtain a non-degenerate alternating $\Z$-module $(G/\Center(G),\omega_G, \Center(G))\leteq\alternatingFunctor(\maxCBAfunctor(G))$. 
	Hence \autoref{rem:HeisenbergOfAlternetingModule} gives a non-degenerate $\Z$-bilinear map $\mu_G\from A\times B\to \Center(G)$ (where actually $A\isom B$) and a \centralByAbelian{} extension 
	\[\begin{tikzcd}[label]
		\HeisenbergFunctor(\mu_G) \ar[:]& 
		1\ar[r] & 
		\Center(G)\ar[r,mono,"{\HeisenbergMono[\mu_G]}"] & 
		\HH(\mu_G) \ar[r,epi,"{\HeisenbergEpi[\mu_G]}"]  & 
		A\times B \ar[r] & 
		1.
	\end{tikzcd}\]
	In this way, we assigned a Heisenberg group $\HH(\mu_G)$ to $G$. This Heisenberg group and the original group $G$ share many properties: the order, isomorphism class of centre and the commutator subgroup, and the nilpotency class. 
	However, unless $G\isom \HH(\mu_G)$, we cannot even expect to have a morphism between the two short exact sequences above, as that would imply $G\isom \HH(\mu_G)$ by the Short Five Lemma for groups \cite[Chapter~I. Lemma~3.1]{MacLane1995-vg}. 
\end{note}

Finally, we compute the Heisenberg group associated to our example nilpotent groups. 

\begin{exmp}[$\extraSPecialD{n}$, $\extraSPecialQ{n}$: Associated Heisenberg groups -- continuing \autoref{exmp:DQhermitian}]\label{exmp:DQmu}
	Recall $Q=\Z/(n)$. 
	Since $M_Q=Q \coset{\alpha}$ and $iM_Q=Q \coset{\beta}$, the $\Z$-bilinear map $\mu_G\from \generate{\coset{\alpha}} \times \generate{\coset{\beta}}\to\generate{\gamma}$ is given by $(\coset{\alpha},\coset{\beta})\mapsto \gamma$. 
	This means that both $\mu_\extraSPecialD{n}$ and $\mu_\extraSPecialQ{n}$ are 
	isomorphic to $\mu_n$ from \autoref{exmp:DHQnotH}, 
	hence $\HH(\mu_\extraSPecialD{n})\isom \HH(\mu_\extraSPecialQ{n})\isom\HH(\mu_n)\isom \extraSPecialD{n}$. 
	In short, the Heisenberg groups associated to $\extraSPecialD{n}$ and to $\extraSPecialQ{n}$ are both isomorphic to $\extraSPecialD{n}$. 
	This demonstrates that this construction needs to be changed to attain the embedding of \autoref{thm:mainHeisenberg}. 
\end{exmp}

\section{Embedding into Heisenberg groups}\label{sec:HeisenbergEmbedding}

\begin{summary}
	We put together the pieces from previous sections to prove the main statement of the paper, \autoref{thm:mainHeisenberg} from \autopageref{thm:mainHeisenberg}: every finite nilpotent group $G$ of class at most $2$ is a subgroup of a suitable Heisenberg group. 
	First, we handle the special case when the centre of $G$ is cyclic by modifying the construction of \autoref{sec:HeisenbergGroups}. 
	More explicitly, we replace the centre of the associated Heisenberg group with a larger cyclic group. 
	For this, we introduce the notion of extended polarisation to generalise the isotropic structure from \autoref{sec:alterntingModules}.
	Finally, we use \autoref{part:mainStructureSubdirect} of \autoref{thm:mainStructure} from \pageref{thm:mainHeisenberg} to reduce the general case of the problem to the previous special case by considering the direct product of the Heisenberg groups obtained there.
\end{summary}

\subsection{Extended polarisation and the induced embedding}

Our goal in to establish a monomorphism $\maxCBAfunctor(G)\to \HeisenbergFunctor(\extended{\mu})$ for a suitable (non-degenerate) $\extended{\mu}$. 
In fact, we will show that $\extended{\mu}=\zeta\circ\mu_G$ works for a suitable embedding $\zeta:\Center(G)\embeds \extended{C}$ where $\mu_G$ is from \autoref{exmp:mu_G}. 
For this, we generalise the notion of isotropic structure from \autoref{lem:complexStructure} used in \autoref{rem:HeisenbergOfAlternetingModule}.

\begin{defn}\label{def:polarisationOfExtensions}
	An \emph{extended polarisation} of a \centralByAbelian{} extension $\epsilon$ is a pair of the following commutative diagrams ($j\in\{1,2\}$)
	\begin{equation}\label{diag:polarisationOfExntensions}
		\begin{tikzcd}[label]
			\epsilon_j\ar[:] & 1\ar{r} &  C_j \ar[r,"\iota_j",dashed]\ar[d,"\kappa_j",dashed] & G_j \ar[r,"\pi_j",dashed]\ar[d,"\gamma_j",dashed] \ar[ddl,dashed,bend left=10,"\zeta_j" near end] & L_j\ar[r,dashed]\ar[d,mono,"\lambda_j",dashed] & 1 \\
			\epsilon\ar[:]& 1\ar{r} & C \ar[r,"\iota"] \ar[d,dashed,"\zeta"] & G \ar["\pi"]{r} & M \ar{r} & 1
			\\
			& & \extended{C} & & L_1\times L_2 \ar[u,iso,dashed,"\lambda"']
		\end{tikzcd}
	\end{equation}
	such that $\epsilon_j$ is a \centralByAbelian{} extension, 
	the map $\lambda\from L_1\times L_2\to M$ given by $(l_1,l_2)\mapsto \lambda_1(l_1)+\lambda_2(l_2)$ is an isomorphism and 
	$\extended{C}$ is an abelian group.
\end{defn}

\begin{lem}\label{lem:G=G_2CG_1}
	Every extended polarisation as in \eqref{diag:polarisationOfExntensions} induces the product decomposition $G=\gamma_2(G_2)\iota(C)\gamma_1(G_1)$.
\end{lem}
\begin{proof}
	Pick $g\in G$ arbitrarily. 
	Set $(l_1,l_2)\leteq \lambda^{-1}(\pi(g))\in L_1\times L_2$. 
	From the surjectivity of $\pi_j$, pick $g_j\in G_j$ such that $\pi_j(g_j)=l_j$. 
	Then 
	$\pi(\gamma_2(g_2)^{-1}g\gamma_1(g_1)^{-1})=
	-\lambda_2(\pi_2(g_2)) + (\lambda_1(l_1)+\lambda_2(l_2)) -\lambda_1(\pi_1(g_1))=
	0$
	using the commutativity of the diagram \eqref{diag:polarisationOfExntensions}. 
	So  $\gamma_2(g_2)^{-1}g\gamma_1(g_1)^{-1}\in \ker(\pi)=\im(\iota)$, hence there is a $c\in C$ such that $\iota(c)=\gamma_2(g_2)^{-1}g\gamma_1(g_1)^{-1}$. Rearranging this gives the decomposition as claimed.
\end{proof}

The next statement is the key to finding embeddings into Heisenberg groups.
\begin{prop}\label{lem:key}
	Every extended polarisation as in \eqref{diag:polarisationOfExntensions} can be completed to a commutative diagram
	\begin{equation}
		\label{diag:key}
		\begin{tikzcd}[label]
			\epsilon\ar[d,dashed,"\exists"]\ar[:]
			& 1\ar{r} 
			& C \ar[r,"\iota"]\ar[d,"\zeta"] 
			& G \ar["\pi"]{r}\ar[d, dashed,"\exists \delta"] 
			& M \ar{r}\ar[d, iso',"\lambda^{-1}"] 
			& 1 
			\\
			\HeisenbergFunctor(\extended{\mu})\ar[:]
			&1\ar{r} 
			& \extended{C} \ar[r,"{\HeisenbergMono[\extended{\mu}]}"] 
			& \HH(\extended{\mu}) \ar["{\HeisenbergEpi[\extended{\mu}]}"]{r} 
			& L_1\times L_2 \ar{r} 
			& 1
		\end{tikzcd}
	\end{equation}
	where $\extended{\mu}$ is defined by  
	\[\begin{tikzcd}
		M\times M \ar[r,"\omega"]
		& C \ar[d,"\zeta"]
		\\
		L_1 \times L_2 \ar[r,dashed,"\extended{\mu}"] \ar[u,mono,"\lambda"]
		& \extended{C}
	\end{tikzcd}\]
	for the alternating $\Z$-bilinear map $\omega$ from \autoref{lem:alterntingFunctor} when applied to $\epsilon$. 
\end{prop}

\begin{rem}\label{rem:deltaMono}
	The Short Five Lemma for groups \cite[Chapter~I. Lemma~3.1]{MacLane1995-vg} shows that $\delta$ is injective if and only if $\zeta$ is.
	In this case, $\HH(\extended{\mu})$ is the external central product of $\extended{C}$ and $G$ amalgamating $C$ along $\iota$ and $\zeta$; 
	in particular, $G$ is isomorphic to a normal subgroup of a Heisenberg group.
\end{rem}

\begin{rem}\label{rem:keyNonDegeneracy}
	The \centralByAbelian{} extension $\epsilon$ is non-degenerate if and only if $\HH(\extended{\mu})$ is non-degenerate.
\end{rem}

\begin{proof}[Proof of \autoref{lem:key}]
	First note that $\extended{\mu}$ is indeed an alternating $\Z$-bilinear map by \autoref{lem:alterntingFunctor}, so $\HH(\extended{\mu})$ is well defined.
	We show that $\delta\leteq (\delta_1,\delta_2,\delta_3)$ satisfies the statement for $j\in\{1,2\}$ and 
	\[\delta_j\from G\to L_j,g\mapsto \pi_j(g_j),\qquad 
	\delta_3\from G\to \extended{C},g\mapsto \zeta_2(g_2)\zeta(c)\zeta_1(g_1),\]
	for any decomposition $g=\gamma_2(g_2)\iota(c)\gamma_1(g_1)$ from \autoref{lem:G=G_2CG_1}.
	The map $\delta_j$ is actually the natural composition of the group morphisms $G\xrightarrow{\pi} M\xrightarrow{\lambda^{-1}} L_1\times L_2\to L_j$. In particular, $\delta_j$ is independent of the choice of the decomposition. 
	To show that $\delta_3$ is independent of the choice of the decomposition, let $\gamma_2(g_2)\iota(c)\gamma_1(g_1)=g=\gamma_2(g_2')\iota(c')\gamma_1(g_1')$.
	Then on one hand, $\pi_j(g_j)=\delta_j(g)=\pi_j(g_j')$ by above, 
	hence by the exactness of $\epsilon_j$, there are $c_j\in C_j$ such that  
	$\iota_1(c_1)=g_1'g_1^{-1}$ and $\iota_2(c_2)=g_2^{-1}g_2'$. 
	On the other hand, rearranging the original equation using $\iota(C)\subseteq \Center(G)$ gives 
	$\iota(cc'^{-1})=\gamma_2(g_2^{-1}g_2')\gamma_1(g_1'g_1^{-1})=
	\iota(\kappa_2(c_2)\kappa_1(c_1))$, 
	hence $cc'^{-1}=\kappa_2(c_2)\kappa_1(c_1)$ as $\iota$ is injective. 
	Putting these together gives 
	\begin{align*}
		\zeta_2(g_2)\zeta(c)\zeta_1(g_1)
		&=\zeta_2(g_2'\iota_2(c_2)^{-1}) \cdot
		\zeta(\kappa_2(c_2)c'\kappa_1(c_1)) \cdot 
		\zeta_1(\iota_1(c_1)^{-1}g_1')
		\\&=\zeta_2(g_2') \cdot 
		( \zeta_2(\iota_2(c_2))^{-1} \zeta(\kappa_2(c_2)) )\cdot 
		\zeta(c') \cdot 
		( \zeta(\kappa_1(c_1)) \zeta_1(\iota_1(c_1))^{-1} ) \cdot 
		\zeta_1(g_1')
		\\&=\zeta_2(g_2')\zeta(c')\zeta_1(g_1')
	\end{align*}
	using commutativity of \eqref{diag:polarisationOfExntensions}. Thus $\delta_3$ is indeed well defined.

	Note that, unlike the other maps, $\delta_3$ is just a map of sets, \emph{not} a group morphism.
	Its failure to be a group morphism is measured by $\extended{\mu}$.  
	Indeed, pick decompositions
	$g=\gamma_2(g_2)\iota(c)\gamma_1(g_1)$ and $g'=\gamma_2(g_2')\iota(c')\gamma_1(g_1')$. 
	Set $x\leteq \omega(\pi(\gamma_1(g_1)),\pi(\gamma_2(g_2')))\in C$ and note that 
	$\iota(x)=[\gamma_1(g_1),\gamma_2(g_2')]$. 
	Use this to find a decomposition of the product as 
	\begin{align*}
		gg'&=
		\gamma_2(g_2)\iota(c)\gamma_1(g_1)\gamma_2(g_2')\iota(c')\gamma_1(g_1')\\&=
		\gamma_2(g_2)\gamma_2(g_2')\iota(cc')[\gamma_1(g_1),\gamma_2(g_2')]\gamma_1(g_1)\gamma_1(g_1')\\&=
		\gamma_2(g_2g_2')\iota(cc'x)\gamma_1(g_1g_1').
		\intertext{Then by definitions and using the commutativity of the diagram,}  
		\delta_3(gg')
		&=\zeta_2(g_2g_2')\zeta(cc'x)\zeta_1(g_1g_1')
		\\&=\zeta_2(g_2)\zeta(c)\zeta_1(g_1)\cdot \zeta_2(g_2')\zeta(c')\zeta_1(g_1')\cdot \zeta(\omega(\pi(\gamma_1(g_1)),\pi(\gamma_2(g_2'))))
		\\&=
		\delta_3(g)\delta_3(g')\extended{\mu}(\pi_1(g_1),\pi_2(g_2'))\\&=
		\delta_3(g)\delta_3(g')\extended{\mu}(\delta_1(g),\delta_2(g')).
		\intertext{This property together with \autoref{rem:Hdef} imply that $\delta$ is indeed a group morphism:} 
		\delta(gg')&=
		(\delta_1(gg'),\delta_2(gg'),\delta_3(gg'))\\&=
		(\delta_1(g)\delta_1(g'),\delta_2(g)\delta_2(g'),\delta_3(g)\delta_3(g')\extended{\mu}(\delta_1(g),\delta_2(g')))\\&=
		(\delta_1(g),\delta_2(g),\delta_3(g))*(\delta_1(g'),\delta_2(g'),\delta_3(g'))\\&=
		\delta(g)*\delta(g').
	\end{align*}
	
	Finally, we check that diagram \eqref{diag:key} is commutative. 
	Indeed, if $c\in C$, then using the decomposition $\iota(c)=\gamma_2(1)\iota(c)\gamma_1(1)$ gives 
	$\delta\circ\iota = (c\mapsto (0,0,\zeta(c))) = \HeisenbergMono[\extended{\mu}]\circ\zeta$ by definitions. 
	Similarly, the decomposition $g=\gamma_2(g_2)\iota(c)\gamma_1(g_1)\in G$ gives the equality 
	$\lambda^{-1}\circ \pi=(g\mapsto 
	\lambda^{-1}(\pi(\gamma_2(g_2)\gamma_1(g_1)))
	=\delta_1(g)+\delta_2(g))=
	\HeisenbergEpi[\extended{\mu}] \circ\delta$.
\end{proof}

\begin{exmp}[$\extraSPecialD{n}$, $\extraSPecialQ{n}$: Heisenberg embeddings --  continuing \autoref{exmp:DQmu}]\label{exmp:DQembeddingToHbad}
	Let $g\in\{\alpha,\beta\}$. 
	We will replace $C$ from \autoref{exmp:DQomega} by  the iterated external central product $C_G^{\cprod} \leteq (\generate{\gamma} \cprod \generate{\alpha})\cprod \generate{\beta}$ where the amalgamations are along the natural inclusions of $\generate{\gamma}\cap \generate{g}$ into $\generate{\gamma}$ and into $\generate{g}$.
	To do so, note that the solid arrows of the following commutative diagram induce an extended polarisation of $\maxCBAfunctor(G)$ where $\zeta_G^{\cprod}$ and $\zeta_g^{\cprod}$ are the natural inclusions into the external central product $C_G^{\cprod}$.
	\begin{equation}\label{diag:DQextendedPolarisation}
		\begin{tikzcd}[label]
			\epsilon_g\ar[:] & 1\ar{r} &  \generate{\gamma}\cap \generate{g} \ar[r,inclusion]\ar[d,inclusion] & \generate{g} \ar[r,"\maximalCBAProj"]\ar[d,inclusion] \ar[ddl,bend left=10,mono, "\zeta_g^{\cprod}"{pos=0.6}]& \generate{\coset{g}}\ar[r]\ar[d,inclusion] & 1 
			\\
			\maxCBAfunctor(G)\ar[:]& 1\ar{r} & \generate{\gamma} \ar[r,inclusion] \ar[d,mono,"\zeta_G^{\cprod}"] & G \ar[d,dotted,mono,"\delta_G^{\cprod}"] \ar["\maximalCBAProj"]{r} & G/\Center(G)\ar[d,identity] \ar{r} & 1
			\\
			\HeisenbergFunctor(\mu_G^{\cprod})\ar[:]& 1\ar[r,dashed] &  C_G^{\cprod} \ar[r,dashed,"{\HeisenbergMono[\mu_G^{\cprod}]}"]& \HH(\mu_G^{\cprod}) \ar[r,dashed,"{\HeisenbergEpi[\mu_G^{\cprod}]}"] & \generate{\coset{\alpha}}\times \generate{\coset{\beta}} \ar[r,dashed] & 1
		\end{tikzcd}
	\end{equation}
	Then \autoref{lem:key} gives an embedding 
	\begin{equation}
		\label{eq:extraSpecialToH}
		\begin{aligned}
			\delta_G^{\cprod}\from G&\embeds \HH(\mu_G^{\cprod}\from \generate{\coset{\alpha}}\times\generate{\coset{\beta}}\to C_G^{\cprod})\\
			\beta^b\gamma^c\alpha^a&\mapsto (\coset{\alpha}^a, \coset{\beta}^b, [(\gamma^c,\alpha^a,\beta^b)])
		\end{aligned}
	\end{equation}
	where the $\Z$-bilinear map $\mu_G^{\cprod}$ is defined by $\zeta_G^{\cprod}\circ \mu_G$, more concretely by $\mu_G^{\cprod}\from (\coset{\alpha}^a,\coset{\beta}^b)\mapsto [(1,\gamma^{ab},1)]$. 
	Up to this point, in all examples both $G\in\{\extraSPecialD{n}, \extraSPecialQ{n}\}$ behaved in the exact same way. 
	However, 
	$\mu_\extraSPecialD{n}^{\cprod}\not\isom  \mu_\extraSPecialQ{n}^{\cprod}$ since  
	$C_\extraSPecialD{n}^{\cprod}\not\isom C_\extraSPecialQ{n}^{\cprod}$, 
	even if \autoref{exmp:DQmu} shows that $\mu_\extraSPecialD{n}\isom \mu_\extraSPecialQ{n}$.
	We investigate the difference between $G=\extraSPecialD{n}$ and $G=\extraSPecialQ{n}$ further.	 
	\begin{enumerate}
		\item For $G=\extraSPecialD{n}$, 
		we have $C_\extraSPecialD{n}^{\cprod} \isom 
		\generate{\gamma} \times \generate{\alpha}\times \generate{\beta} \isom 
		\Z_n^3$
		as both amalgamations are along the trivial $\generate{\gamma}\cap \generate{\alpha}=\generate{\gamma}\cap \generate{\beta}=1$. 
		Thus \eqref{eq:extraSpecialToH} translates to  
		\begin{equation}\label{eq:DtoHold}
			\begin{aligned}
				\delta_\extraSPecialD{n}^{\cprod}\from G&\embeds \HH(\mu_\extraSPecialD{n}^{\cprod} \from \Z_n\times \Z_n\to \Z_n\times \Z_n^2)\\
				\beta^b\gamma^c\alpha^a&\mapsto ([a], [b], ([c], ([a],[b])))
			\end{aligned}
		\end{equation}
		where $\mu_\extraSPecialD{n}^{\cprod}$ is given by $([a],[b])\mapsto ([ab],([0],[0]))$.

		\item 
		For $G=\extraSPecialQ{n}$, 
		note that the map $C_\extraSPecialQ{n}^{\cprod}\to 
		\generate{\alpha}\times \generate{\alpha^{n-1}\beta}$ taking $([\gamma^c,\alpha^a,\beta^b])$ to $(\alpha^{a+nc-(n-1)b}, (\alpha^{n-1}\beta)^b)$ is an isomorphism, because this time  
		we amalgamate along $\generate{\gamma}\cap \generate{\alpha}=\generate{\gamma}\cap\generate{\beta}=\generate{\gamma}$ in $C_\extraSPecialQ{n}^{\cprod}$. 
		This shows that $C_\extraSPecialQ{n}^{\cprod}\isom \Z_n^2 \times \Z_n$, 
		hence \eqref{eq:extraSpecialToH} produces the embedding
		\begin{equation}\label{eq:QtoHold}
			\begin{aligned}
				\delta_\extraSPecialQ{n}^{\cprod}\from G&\embeds \HH(\mu_\extraSPecialQ{n}^{\cprod} \from \Z_n\times \Z_n\to \Z_{n^2}\times \Z_n)\\
				\beta^b\gamma^c\alpha^a&\mapsto ([a], [b], ([a+b+nc-nb], [b]))
			\end{aligned}
		\end{equation}
		where $\mu_\extraSPecialQ{n}^{\cprod}$ is defined by $([a],[b])\mapsto ([nab], [0])$.
	\end{enumerate}
	We will show in \autoref{exmp:DQembeddingToHoptimal} that (by replacing the notion of the central product suitably) the second direct factor of $C_G^{\cprod}$ can be omitted yielding cyclic centre and keeping the embedding as needed for \autoref{thm:mainHeisenberg}.
\end{exmp}

\subsection{Improved embedding and the proof of \autoref{thm:mainHeisenberg}}

To fix the issue from the end of \autoref{exmp:DQembeddingToHbad}, we need the following elementary statement.
\begin{lem}\label{lem:Zextension}
	Let $A$ be a finite abelian group, 
	let $C$ and $Z$ be finite cyclic groups 
	together with an injection $\iota\from C\embeds A$ and a morphism $\kappa\from C\to Z$. 
	Then there exist a cyclic group $P$ of order $\lcm(\exp(A), |Z|)$ and morphisms to $P$ making the following diagram commutative.
	\[\begin{tikzcd}
		C \ar[r,mono,"\iota"] \ar[d,"\kappa"] & A \ar[d,dashed,"\exists\phi"] \\
		Z \ar[r,mono,dashed,"\exists\theta"] & P
	\end{tikzcd}\]
\end{lem}
\begin{rem}\label{rem:fprod}
	Note that if $A$ is cyclic and $\kappa$ is injective, then the resulting $P$ is the external central product $Z\cprod A$ amalgamating the largest possible subgroups containing the images of $C$. 
	We introduce the notation $P\leteq Z\fprod A$ and think of it as a kind of product of $Z$ and $A$ producing a cyclic group in which the images of $C$ are identified. 
	Note that $P$ is given only up to isomorphism and the maps $\phi$ and $\theta$ are not canonical in any way.
\end{rem}

\begin{proof}[Proof of \autoref{lem:Zextension}]
	Without loss of generality, we may assume that $C=\Z_m$, $Z=\Z_n$ and look for $P$ of the form $\Z_l$ where $l=\lcm( \exp(A),|Z|)$.
	
	First, we prove the case $A=\Z_k$. 
	The map $\kappa$ is defined by some $b\in \Z$ such that $\kappa(1+m\Z)=\frac{n}{m}b+n\Z$. 
	Similarly, $\iota$ is given by $\iota(1+m\Z)=\frac{k}{m}a+k\Z$ for some $a\in\Z$. 
	Since $\iota$ is injective, we have $\gcd(a,m)=1$, hence 
	we may pick $x\in \Z$ so that $ax\equiv b \pmod{m}$. 
	Define $\phi\from (i+k\Z)\mapsto i\frac{l}{k}x+l\Z$
	and an injection $\theta\from (j+n\Z)\mapsto j\frac{l}{n}+l\Z$. 
	A short computation shows that $\phi\circ\iota = \theta\circ\kappa$ with these definitions as required in this case.

	In the general case, write $A=\prod_{K\in \mathcal{S}} K$ for some suitable set $\mathcal{S}$ of cyclic subgroups of $A$ of prime power order. 
	For $K\in \mathcal{S}$, let $\pi_K\from A\to K$ be the natural projection, and write $o(K)\leteq |\im(\pi_K\circ\iota)|$. 
	For every prime divisor $p$ of $m$, pick a $K_p\in \mathcal{S}$ so that $o(K_p)=\max\{o(K):K\in \mathcal{S},p\divides |K|\}$. 
	Define the composition 
	\[\bar\iota\from C\xrightarrow{\iota} A\xsurjects{\pi} \prod_{p\mid m} K_p\isom \Z_d\xembeds{\upsilon}  \Z_k\] where 
	the $\pi$ is the product map $\prod_{p\mid m}\pi_{K_p}$,  
	the isomorphism is given by the Chinese remainder theorem, 
	and $\upsilon$ is any embedding where $d\divides k\leteq \exp(A)$. 
	The morphism $\bar\iota$ is injective, because 
	$|\im(\bar\iota)|  
	= \lcm\{|K_p|:p\mid m\}
	= \lcm\{o(K):K\in \mathcal{S}\} 
	=m
	$. 
	So replacing $\iota$ by $\bar\iota$ reduces to the special case $A=\Z_k$ discussed above.
\end{proof}

\begin{exmp}[$\extraSPecialD{n}$, $\extraSPecialQ{n}$: Improved Heisenberg embeddings --  continuing \autoref{exmp:DQembeddingToHbad}]\label{exmp:DQembeddingToHoptimal}
	Replace the external central product $C_G^{\cprod}= (\generate{\gamma} \cprod \generate{\alpha}) \cprod \generate{\beta}$
	from \autoref{exmp:DQembeddingToHbad} 
	by $C_G^{\fprod}\leteq (\generate{\gamma} \fprod \generate{\alpha})\fprod \generate{\beta}$ from \autoref{rem:fprod}.
	More concretely, apply \autoref{lem:Zextension} successively to the diagrams
	\begin{equation}\label{diag:DQzetaConstruction}
		\begin{tikzcd}
			\generate{\gamma}\cap\generate{\alpha} \ar[r,inclusion,"\iota_\alpha"'] \ar[d,inclusion,"\kappa_\alpha"'] & \generate{\alpha} \ar[d,dashed,"\exists\phi_\alpha"] &
			\generate{\gamma}\cap\generate{\beta} \ar[r,inclusion,"\iota_\beta"',dashed] \ar[d,inclusion,"\kappa_\beta"',dashed] & \generate{\beta} \ar[d,dotted,"\exists\phi_\beta"]
			\\
			\generate{\gamma} \ar[r,mono,dashed,"\exists\theta_\alpha"] & \generate{\gamma}\fprod\generate{\alpha} \ar[r,equal]&
			\generate{\gamma}\fprod\generate{\alpha} \ar[r,mono,dotted,"\exists\theta_\beta"] & C_G^{\fprod}
		\end{tikzcd}
	\end{equation}
	where $\kappa_\beta$ is the restriction of $\theta_\alpha$ to ${\generate{\gamma}\cap \generate{\beta}}$. 
	\begin{enumerate}
		\item Consider the case $G=\extraSPecialD{n}$. 
		We can take $\generate{\gamma}\fprod\generate{\alpha}=\generate{\gamma}$ together with  $\phi_\alpha\from \alpha\mapsto \gamma$ and $\theta_\alpha\from \gamma\mapsto \gamma$. (In short, we identify $\gamma\sim\alpha$ in $\generate{\gamma}\fprod\generate{\alpha}$.)
		Similarly, we can take $ C_G^{\fprod}=\generate{\gamma}\fprod\generate{\beta} = \generate{\gamma}$ by identifying $\gamma\sim\beta$. 
		Then $\mu_\extraSPecialD{n}^{\fprod}\from \generate{\coset{\alpha}}\times \generate{\coset{\beta}}\to \generate{\gamma}$ is given by $(\coset{\alpha},\coset{\beta})\mapsto \gamma$, 
		and the resulting Heisenberg embedding $\delta_\extraSPecialD{n}^{\fprod}\from \extraSPecialD{n}\embeds \HH(\mu_\extraSPecialD{n}^{\fprod})$ recovers \eqref{eq:DisoH}, because $\mu_\extraSPecialD{n}^{\fprod}$ and $\mu_n$ are isomorphic. 
		In this way, \eqref{eq:DisoH} can be considered as an improved version of \eqref{eq:DtoHold}.

		\item For $G=\extraSPecialQ{n}$, 
		we can take $\generate{\gamma}\fprod \generate{\alpha}=\generate{\alpha}$ by identifying $\gamma\sim \alpha^n$, 
		and then 
		$C_G^{\fprod}=\generate{\alpha}\fprod\generate{\beta} = \generate{\alpha}$ by $\alpha\sim\beta$. 
		Now $\mu_\extraSPecialQ{n}^{\fprod}\from \generate{\coset{\alpha}}\times \generate{\coset{\beta}}\to \generate{\alpha}$ is given by 
		$(\coset{\alpha},\coset{\beta})\mapsto \alpha^n$. 
		More concretely, this $\mu_\extraSPecialQ{n}^{\fprod}$ is isomorphic to $\mu_n^{\fprod}\from \Z_n\times\Z_n\to \Z_{n^2}$ defined by $\mu_n^{\fprod}\from ([a], [b])\mapsto [nab]$.  
		Hence the resulting Heisenberg embedding takes the form 
		\begin{equation}
			\label{eq:QembedsToH}
			\begin{aligned}
				\delta_\extraSPecialQ{n}^{\fprod}\from \extraSPecialQ{n}&\embeds \HH(\mu_n^{\fprod}\from \Z_n\times\Z_n\to \Z_{n^2})\\
				\beta^b\gamma^c\alpha^a&\mapsto ([a], [b], [a+b+nc])
			\end{aligned}
		\end{equation} 
		improving \eqref{eq:QtoHold}.
	\end{enumerate}
	Note that unlike $C_G^{\cprod}$, our new $C_G^{\fprod}$ is indeed \emph{cyclic} which fixes the issue from the end of \autoref{exmp:DQembeddingToHbad}.
\end{exmp}

The next statement generalises the construction of \autoref{exmp:DQembeddingToHoptimal}: using \autoref{lem:Zextension}, we replace \eqref{diag:DQzetaConstruction} by \eqref{diag:hatCconstruction} to produce a suitable extended polarisation in the cyclic centre case.

\begin{prop} \label{thm:maximalCbACyclicCentreEmbedsToHeisenberg}
	For every finite nilpotent group $G$ of class at most~$2$ with cyclic centre, there exists a monomorphism 
	\begin{equation}\label{diag:embeddingToH}
		\begin{tikzcd}[label]
			\maxCBAfunctor(G)\ar[:] \ar[d,"\exists f",dashed]
			& 1\ar[r] 
			& \Center(G) \ar[r,inclusion]\ar[d,mono,"\exists \zeta",dashed] 
			& G \ar[r,epi,"\maximalCBAProj"]\ar[d, mono,"\exists \delta",dashed] 
			& G/\Center(G) \ar{r}\ar[d, iso',"\exists \nu",dashed] 
			& 1 
			\\
			\HeisenbergFunctor(\extended{\mu})\ar[:]&1\ar[r] & \extended{C} \ar[r,"{\HeisenbergMono[\extended{\mu}]}"] & \HH(\extended{\mu}) \ar[r,"{\HeisenbergEpi[\extended{\mu}]}"] & A\times A \ar[r] & 1
		\end{tikzcd}
	\end{equation}
	of non-degenerate \centralByAbelian{} extensions for a suitable $\Z$-bilinear map $\extended{\mu}\from A\times A\to  \extended{C}$ where 
	$\exp(A)= |G'|$ and 
	$\extended{C}$ is cyclic of order satisfying $|\extended{C}| \divides\exp(G)$.	
\end{prop}
\begin{rem}
	Actually, the Heisenberg group from the statement can be replaced by a canonical one as follows. 
	Let $\Hom(A,\extended{C})$ denote the set of $\Z$-linear maps from $A$ to $\extended{C}$. 
	Define a $\Z$-bilinear map $\nu\from \Hom(A,\extended{C})\times A\to \extended{C}$ by $(\alpha, a)\mapsto \alpha(a)$, and write $\HH(A,\extended{C})\leteq \HH(\nu)$ for the corresponding Heisenberg group. 
	Then the map $\HH(\extended{\mu})\to \HH(A,\extended{C})$ given by $(a,a',c)\mapsto (x\mapsto \extended{\mu}(a,x), a', c)$ is an isomorphism because $\extended{\mu}$ is non-degenerate and $\exp(A)\divides |C|$. 
	In particular, every $G$ as above is isomorphic to a normal subgroup of $\HH(A, \extended{C})$ of index at most $\exp(G)/|\Center(G)|$. 
\end{rem}

\begin{proof}[Proof of \autoref{thm:maximalCbACyclicCentreEmbedsToHeisenberg}]
	By \autoref{rem:dictionary}, we see that \autoref{lem:complexStructure} is applicable to $(G/\Center(G),\omega,\Center(G))\leteq \alternatingFunctor(\maxCBAfunctor(G))$ giving  an isotropic structure $G/\Center(G)=L_1\oplus L_2$. 
	Note that $L_2=iL_1\isom L_1$.  
	Since $L_j$ is isotropic, $\maximalCBAProj^{-1}(L_j)$ is abelian by \autoref{rem:dictionary} for $j\in\{1,2\}$. 
	
	Consider the inclusion maps (indicated by solid arrows) of the following diagram.
	\begin{equation}\label{diag:hatCconstruction}
		\begin{tikzcd}
			Z(G)\cap \maximalCBAProj^{-1}(L_1) \ar[r,inclusion,"\iota_1"'] \ar[d,inclusion,"\kappa_1"'] 
			& \maximalCBAProj^{-1}(L_1) \ar[d,dashed,"\phi_1"' near start] \ar[drr,dotted,"\zeta_1"' near start, bend right=0]
			& \Center(G)\cap \maximalCBAProj^{-1}(L_2) \ar[dll,mono,"\iota"' near start, "\supseteq"'{sloped, near start},bend left=0] \ar[d,dashed,mono,"\kappa_2" near start] \ar[r,inclusion,"\iota_2"']
			& \maximalCBAProj^{-1}(L_2) \ar[d,dotted,"\phi_2=\zeta_2"']
			\\
			\Center(G) \ar[r,mono,dashed,"\theta_1" near end] \ar[rrr,mono,bend right=15,dotted,"\zeta"] 
			& P \ar[r,equal] 
			& P \ar[r,dotted,mono,"\theta_2" near start] 
			& \extended{C} 
		\end{tikzcd}
	\end{equation}

	Applying \autoref{lem:Zextension} to the inclusions $\iota_1$ and $\kappa_1$ gives a cyclic group $P=\Center(G)\fprod \maximalCBAProj^{-1}(L_1)$ together with a morphism $\phi_1$ and an injection $\theta_1$. 
	Then we can apply \autoref{lem:Zextension} once again to the inclusion $\iota_2$ and the composition $\kappa_2=\theta_1\circ \iota$ yielding yet another cyclic group $\extended{C}=P\fprod \maximalCBAProj^{-1}(L_2)$ with a morphism $\phi_2$ and an injection $\theta_2$.
	By construction, $\theta_1$ and $\theta_2$ are both injective, hence so is 
	$\zeta\leteq \theta_2\circ \theta_1$. 
	Set $\zeta_1\leteq \theta_2\circ \phi_1$ and $\zeta_2\leteq \phi_2$. 
	Note that these definitions make the whole diagram \eqref{diag:hatCconstruction} above commutative by construction. 
		
	The maps above (for $j\in\{1,2\}$) give the extended polarisation 
	\begin{equation}\label{diag:extendedPolarisationConstruction}
		\begin{tikzcd}[label]
			\epsilon_j\ar[:] & 1\ar{r} &  \Center(G)\cap \maximalCBAProj^{-1}(L_j) \ar[r,inclusion]\ar[d,inclusion] & \maximalCBAProj^{-1}(L_j) \ar[r,"\pi_j",dashed]\ar[d,inclusion] \ar[ddl,bend left=10,"\zeta_j" near end]& L_j\ar[r]\ar[d,mono,inclusion] & 1 
			\\
			\maxCBAfunctor(G)\ar[:]& 1\ar{r} & \Center(G) \ar[r,inclusion] \ar[d,"\zeta",mono] & G \ar["\maximalCBAProj"]{r} & G/\Center(G)\ar[d,identity] \ar{r} & 1
			\\
			& & \extended{C} & & L_1\times L_2
		\end{tikzcd}
	\end{equation}
	of $\maxCBAfunctor(G)$ where $\pi_j\leteq \maximalCBAProj|_{\maximalCBAProj^{-1}(L_j)}$. 
	Then \autoref{lem:key} gives the diagram \eqref{diag:embeddingToH} upon setting $A\leteq L_1\isom L_2$.
	
	Finally, we check that the properties from the statement hold. 
	Note $\delta$ is indeed injective by \autoref{rem:deltaMono} because $\zeta$ is injective by construction. 
	The fact that $\exp(A)=\exp(L_1\times L_2)=\exp(G/\Center(G))=|(G')|$ follows from \autoref{cor:decompositionCyclicDeriverGroup}. 
	Since $G$ is of nilpotency class at most $2$, we have $G'\subseteq \Center(G)$, so $|G'|\divides |\Center(G)|$. 
	By \autoref{lem:Zextension}, 
	$|\extended{C}|$ is the least common multiple of 
	$|\Center(G)|$, $\exp(\maximalCBAProj^{-1}(L_1))$ and $\exp(\maximalCBAProj^{-1}(L_2))$, hence  
	$|\extended{C}| \divides \exp(G)$.
\end{proof}

We are ready to prove the general case, which in fact is a generalisation of \autoref{thm:mainHeisenberg}.

\begin{thm}\label{thm:embedToH}
	For every finite nilpotent group $G$ of class at most~$2$, there exist  
	non-degenerate $\Z$-bilinear maps 
	$\mu\from A\times B\to C$ 
	and $\mu_i\from A_i\times A_i\to C_i$ for $1\leq i\leq d(\Center(G))$, 
	and monomorphisms 
	\begin{equation}\label{eq:embetToH}
	\begin{tikzcd}[label]
		\maxCBAfunctor(G)\ar[:] \ar[d,"\exists f",mono,dashed]& 
		1\ar[r] & 
		\Center(G) \ar[r,central,inclusion]\ar[d,mono,"\exists \zeta",dashed] & 
		G \ar[r,epi,"\maximalCBAProj"]\ar[d, mono,"\exists \delta",dashed] & 
		G/\Center(G) \ar{r}\ar[d,mono,"\exists \nu",dashed] & 
		1 
		\\
		\HeisenbergFunctor(\mu)\ar[:]  \ar[d,dashed,mono,"\exists"] & 
		1\ar[r] & 
		C \ar[r,inclusion] \ar[d,dashed,inclusion,"\exists"']&
		\HH(\mu) \ar[r,"{\HeisenbergEpi[\mu]}"]  \ar[d,dashed,inclusion,"\exists"'] & 
		A\times B \ar[r]  \ar[d,dashed,mono,"\exists"] & 
		1
		\\
		\prod_i\HeisenbergFunctor(\mu_i)\ar[:] &
		1\ar[r] & 
		\prod_i C_i \ar[r,inclusion] & 
		\prod_i \HH(\mu_i) \ar[r,"{\prod_i \HeisenbergEpi[\mu_i]}"] & 
		\prod_i A_i\times A_i \ar[r]   & 
		1
	\end{tikzcd}
	\end{equation}
	of \centralByAbelian{} extensions such that any prime divisor of the order of any group above also divides $|G|$, and 
	\begin{align}
		d(G/\Center(G))\geq \max\{ d(A), d(B), d(A_i^2)\},\qquad
		d(\Center(G))=d(C)\geq d(C_i)=1,
		\label{eq:Droperties}
		\\
		\exp(G/\Center(G))\divides \exp(A\times B) \divides \exp(G') \divides \exp(\Center(G)) \divides \exp(C)\divides \exp(G).
		\label{eq:ExpProperties}
	\end{align}
\end{thm}

\begin{rem}
	The monomorphism $\zeta$ shows that $d(C)$ is as small as possible. 
	On the other hand, $\nu$ gives $d(G/\Center(G))\leq d(A\times B)\leq 2 d(G/\Center(G))$. 
	The lower bound is attained in the case $d(C)\leq 1$, see \autoref{thm:maximalCbACyclicCentreEmbedsToHeisenberg}. 
	It is a natural question to ask for the smallest possible value of $d(A\times B)$ in general. To give a better upper bound than above, one may need to develop some version of \autoref{lem:alternatingSmith} for matrices with entries from $\Z^n$. 
\end{rem}

\begin{proof}[Proof of \autoref{thm:embedToH}]	
	Using \autoref{prop:subdirectProductFinite}, write $\phi\from G\embeds \prod_{i=1}^n G_i$ as a subdirect product where $n\leteq d(\Center(G))$ and  each $\Center(G_i)$ is cyclic. 
	Then $G_i$ is nilpotent of class at most~$2$ as this class of groups is closed under taking quotients, so 
	\autoref{thm:maximalCbACyclicCentreEmbedsToHeisenberg} gives $f_i=(\zeta_i,\delta_i,\nu_i)\from \maxCBAfunctor(G_i)\to \HeisenbergFunctor(\mu_i)$ for some $\mu_i\from A_i\times B_i\to C_i$ (where $B_i=A_i$). 
	Let $\hat A\leteq \prod_{i=1}^n A_i$, $\hat B\leteq \prod_{i=1}^n B_i$, $\hat C\leteq \prod_{i=1}^n C_i$ 
	and set $\hat \mu\leteq \prod_{i=1}^n \mu_i\from \hat A\times \hat B\to \hat C$, which is a non-degenerate $\Z$-bilinear map by construction.
	We obtain the following diagram of \centralByAbelian{} extensions where $\prod_i$ is a shorthand for $\prod_{i=1}^n$.
	\[\begin{tikzcd}[label]
		\maxCBAfunctor(G)\ar[:]  \ar[d,mono,"\text{\autoref{prop:subdirectProductFinite}}"'] & 
		1\ar[r] & 
		\Center(G) \ar[r,inclusion]\ar[d,mono,"\phi|_{\Center(G)}"] & 
		G \ar[r,epi,"\maximalCBAProj"]\ar[d, mono,"\phi"] & 
		G/\Center(G) \ar{r}\ar[d, mono, "{[\phi]}"] & 
		1 
		\\
		\prod_i\maxCBAfunctor(G_i)\ar[:] \ar[d,mono,"\prod f_i","\text{\autoref{thm:maximalCbACyclicCentreEmbedsToHeisenberg}}"'] & 
		1\ar[r] & 
		\prod_i\Center(G_i) \ar[r,inclusion]\ar[d,mono,"\prod \zeta_i"] & 
		\prod_i G_i \ar[r,epi,"\prod \maximalCBAProj"]\ar[d, mono,"\prod \delta_i"] & 
		\prod_i G_i/\Center(G_i) \ar{r}\ar[d, iso',"\prod \nu_i"] & 
		1 
		\\
		\prod_i\HeisenbergFunctor(\mu_i)\ar[:] \ar[d,iso] &
		1\ar[r] & 
		\prod_i C_i \ar[r,inclusion]  \ar[d,identity]  & 
		\prod_i \HH(\mu_i) \ar[r,"{\prod_i \HeisenbergEpi[\mu_i]}"]\ar[d,iso] & 
		\prod_i A_i\times B_i \ar[r]  \ar[d,iso]  & 
		1
		\\
		\HeisenbergFunctor(\hat\mu)\ar[:] &
		1\ar[r]& 
		\hat C \ar[r,inclusion]  & 
		\HH(\hat\mu) \ar[r,"{\HeisenbergEpi[\hat\mu]}"] & 
		\hat A\times \hat B \ar[r]   & 
		1
	\end{tikzcd}\]
	Denote the resulting monomorphism $\maxCBAfunctor(G)\to \HeisenbergFunctor(\hat \mu)$ by $\hat f=(\hat \zeta,\hat \delta,\hat \nu)$. 
	The groups $A$ and $B$ may have more generators than stated, so we take a suitable subobject of $\HeisenbergFunctor(\hat\mu)$.
	Let $A\leq \hat A$ be the image of $G/\Center(G)\xrightarrow{\hat \nu} \hat A\times \hat B\to \hat A$, 
	and $B\leq\hat B$ be that of $G/\Center(G)\xrightarrow{\hat \nu} \hat A\times \hat B\to \hat B$. 
	Then $d(A)$ and $d(B)$ are at most $d(G/\Center(G))$. 
	Let $C\leteq \generate{\hat\zeta(\Center(G)), \hat\mu(A,B)}\leq \hat C$. 
	Then $d(\Center(G)) = d(\hat\zeta(\Center(G))) \leq d(C) \leq d(\hat C)=\sum_{i=1}^n d(C_i)\leq n=d(\Center(G))$, hence comparing the two ends gives $d(C)=d(\Center(G))$. 
	
	Define $\mu\from A\times B\to C, (a,b)\mapsto \hat\mu(a,b)$, a $\Z$-bilinear map. 
	The image of $\hat f$ lies in $\HeisenbergFunctor(\mu)$ by definition, so restricting the domain to $\HeisenbergFunctor(\mu)$ gives a map $f=(\zeta,\delta,\nu)\from \maxCBAfunctor(G)\to \HeisenbergFunctor(\mu)$, i.e. $\hat f = (\HeisenbergFunctor(\mu)\embeds\HeisenbergFunctor(\hat\mu)) \circ f$ for the natural inclusion map. We show that this $f$ satisfies the statement.

	We check that $\mu$ is non-degenerate. 
	Pick $0\neq a\in A$ and write $a=(a_1,\dots,a_n)\in \prod_{i=1}^n A_i$. 
	Without loss of generality, we may assume that $a_1\neq 0$. 
	Then by the non-degeneracy of $\mu_1$, there is a $b_1'\in B_1$ such that $0\neq \mu_1(a_1,b_1')\in C_1$. 
	By the diagram above, there is $g_1'\in G_1$ such that 
	$\nu_1(g_1'\Center(G_1))=(0,b_1')$. 
	As $\phi$ is a subdirect product, there is $g'\in G$ such that the first factor of $[\phi](g'\Center(G))$ is $g_1'\Center(G_1)$. 
	Write $b'=(b_1',\dots,b_n')$ for the image of $g'\Center(G)$ under $G/\Center(G)\xrightarrow{\hat \nu} \hat A\times \hat B\to \hat B$. 
	By construction, $b_1'$ coincides with the choice above.
	By definition, $b\in B$, and 
	$\mu(a,b) = (\hat\mu_1(a_1,b_1'),\dots,\hat\mu_n(a_n,b_n'))\neq 0$ as the first factor is non-trivial by construction.
	This argument remains valid when the roles of $A$ and $B$ are swapped, hence $\mu$ is non-degenerate.

	Consider the statement on the exponents.
	Since  $\nu$ is a monomorphism of abelian groups, we have $\exp(G/\Center(G))\divides \exp(A\times B)$.
	Next, for every $1\leq i\leq d(\Center(G))$, we have  
	$\exp(A_i\times B_i) \divides \exp(G_i') \divides \exp(G')$ 
	using \autoref{thm:maximalCbACyclicCentreEmbedsToHeisenberg} and the fact that as $G_i$ is a quotient of $G$. 
	Thus  $\exp(A\times B)\divides \exp(\hat A\times \hat B)=\lcm\{\exp(A_i\times B_i):1\leq i\leq n\}\divides \exp(G')$. 
	Since $G$ is nilpotent of class at most~$2$, we have $G'\subseteq \Center(G)$, so $\exp(G')\divides \exp(\Center(G))$. 
	The embedding $\zeta\from \Center(G)\embeds C$ shows that $\exp(\Center(G))\divides \exp(C)$. 
	Once again using \autoref{thm:maximalCbACyclicCentreEmbedsToHeisenberg}, 
	we see that $\exp(C_i)\divides \exp(G_i)\divides \exp(G)$ as $G_i$ is a quotient of $G$.
	Then $\exp(C)\divides \exp(\hat C)=\lcm\{\exp(C_i):1\leq i\leq n\}\divides \exp(G)$ as stated.
	
	Finally, by construction, $A_i$ and $C_i$ have been obtained from $G$ by taking quotients and subgroups, so no new prime divisor has been introduced.
\end{proof}

\begin{proof}[Proof of {\autoref{thm:mainHeisenberg}}]
	The embedding $G\embeds\HH(\mu\from A\times B\to C)$ with the properties from the statement is given by \autoref{thm:embedToH} as a vertical map of \eqref{eq:embetToH}. 
\end{proof}

\printbibliography

@book{Suzuki2,
  author    = {Suzuki, Michio},
  title     = {Group theory II},
  publisher = {Springer-Verlag Berlin Heidelberg},
  series    = {Grundlehren der mathematischen Wissenschaften},
  edition   = {1},
  year      = {1982},
  isbn	    = {9783642868870}
}

@article{2generatedClassification,
  title={Two generator $p$-groups of nilpotency class $2$ and their conjugacy classes},
  author={Azhana Ahmad and Arturo Magidin and Robert Fitzgerald Morse},
  journal={Publicationes Mathematicae Debrecen},
  year={2012},
  volume={81},
  pages={145-166}
}

@article{Higman,
    author = {Higman, Graham},
    title = "{Enumerating $p$-Groups. I: Inequalities}",
    journal = {Proceedings of the London Mathematical Society},
    volume = {s3-10},
    number = {1},
    pages = {24-30},
    year = {1960},
    month = {01},
    issn = {0024-6115},
    doi = {10.1112/plms/s3-10.1.24},
    url = {https://doi.org/10.1112/plms/s3-10.1.24}
}

@article{Sims,
    author = {Sims, Charles C.},
    title = "{Enumerating p-Groups}",
    journal = {Proceedings of the London Mathematical Society},
    volume = {s3-15},
    number = {1},
    pages = {151-166},
    year = {1965},
    month = {01},
    issn = {0024-6115},
    doi = {10.1112/plms/s3-15.1.151},
    url = {https://doi.org/10.1112/plms/s3-15.1.151}
}

@misc{guld2019finiteDnilpotent,
	title={Finite subgroups of the birational automorphism group are `almost' nilpotent}, 
	author={Attila Guld},
	year={2019},
	eprint={1903.07206},
	archivePrefix={arXiv},
	primaryClass={math.AG}
}

@misc{guld2020finite2nilpotent,
      title={Finite subgroups of the birational automorphism group are 'almost' nilpotent of class at most two}, 
      author={Guld, Attila},
      year={2020},
      eprint={2004.11715},
      archivePrefix={arXiv},
      primaryClass={math.AG}
}

@article{prokhorov2016jordan,
  title={Jordan property for Cremona groups},
  author={Prokhorov, Yuri and Shramov, Constantin},
  journal={American Journal of Mathematics},
  volume={138},
  number={2},
  pages={403--418},
  year={2016},
  archivePrefix={arXiv},
  primaryClass={math.AG}
}

@article{Riera4dim,
	title={Which finite groups act smoothly on a given 4-manifold?},
	author={Mundet i Riera, Ignasi and Sáez--Calvo, Carles},
	journal={Transactions of the American Mathematical Society},
	volume={375},
	number={02},
	pages={1207--1260},
	year={2022},
	eprint={1901.04223},
	archivePrefix={arXiv},
}

@Book{jacobson1985basic,
	author = {Jacobson, Nathan},
	title = {Basic algebra},
	publisher = {W.H. Freeman},
	year = {1985},
	address = {New York},
	isbn = {0-7167-1480-9}
}

@article{brady_bryce_cossey_1969, 
	title={On certain abelian-by-nilpotent varieties}, 
	volume={1}, 
	DOI={10.1017/S0004972700042325}, 
	number={3}, 
	journal={Bulletin of the Australian Mathematical Society}, 
	publisher={Cambridge University Press}, 
	author={Brady, J.M. and Bryce, R.A. and Cossey, John}, 
	year={1969}, 
	pages={403–416}
}

@misc{Magidin,
	title={Bilinear maps and central extensions of abelian groups}, 
	author={Magidin, Arturo},
	year={1998},
	eprint={9802066},
	archivePrefix={arXiv},
	primaryClass={math.GR}
}

@phdthesis{phd,
	title={Jordan Type Problems via Class 2 Nilpotent and Twisted Heisenberg Groups},
	author={Szabó, Dávid R.},
	year={2021},
	school={Central European University},
	url= {https://www.etd.ceu.edu/2022/szabo_david.pdf}
}

@misc{nilpotentJordanHomeo,
	author = {Csikós, Balázs and Pyber, László and Szabó, Endre},
	title = {Finite subgroups of the homeomorphism group of a compact topological manifold are almost nilpotent},
	year = {2022},
	eprint={2204.13375},
	archivePrefix={arXiv},
	primaryClass={math.GT}
}

@article{ProkhorovShramov2014,
	doi = {10.1112/s0010437x14007581},
	url = {https://doi.org/10.1112%2Fs0010437x14007581},
	year = 2014,
	month = {09},
	publisher = {Wiley},
	volume = {150},
	number = {12},
	pages = {2054--2072},
	author = {Yuri Prokhorov and Constantin Shramov},
	title = {Jordan property for groups of birational selfmaps},
	journal = {Compositio Mathematica}
}

@misc{CsikosMundetPyberSzabo,
	title={On the number of stabilizer subgroups in a finite group acting on a manifold}, 
	author={Balázs Csikós and {M}undet i Riera, Ignasi and László Pyber and Endre Szabó},
	year={2021},
	eprint={2111.14450},
	archivePrefix={arXiv},
	primaryClass={math.GT}
}

@BOOK{MacLane1995-vg,
	title     = {Homology},
	author    = {MacLane, Saunders},
	publisher = {Springer},
	series    = {Classics in Mathematics},
	edition   = {1995},
	month     = {02},
	year      = {1995},
	address   = {Berlin, Germany},
	language  = {en}
}

\bigskip
\noindent
\textsc{Eötvös Loránd University, Faculty of Science}, Pázmány Péter sétány 1/A,
H-1117, Budapest, Hungary

\smallskip

\noindent
\textsc{Alfréd Rényi Institute of Mathematics}, Reáltanoda u. 13–15,
H–1053, Budapest, Hungary 

\smallskip

\noindent
E-mail address: \email{szabo.r.david@gmail.com}
\end{document}